\documentclass[english]{amsart}
\usepackage[english]{babel}

\usepackage{hyperref}
\usepackage{verbatim}
\usepackage{amsfonts,amssymb,mathtools}
\usepackage{amsmath}
\usepackage{bbm}
\usepackage{hyperref}

\usepackage{graphicx,xcolor}
\definecolor{green}{RGB}{89,169,58}
\definecolor{red}{RGB}{224,61,42}
\definecolor{blue}{RGB}{63,78,181}

\graphicspath{{Figures/}}
\usepackage{subcaption,tikz}
\usetikzlibrary{arrows}

\usepackage[style=numeric,maxcitenames=2,maxbibnames=10,doi=false,isbn=false,url=false,natbib=true,giveninits=true,hyperref,bibencoding=utf8,backend=biber]{biblatex}
\AtEveryBibitem{%
  \clearfield{note}%
  \clearfield{issn}%
  \clearlist{language}%
  }

\newcommand{\N}{\ensuremath{\mathbb{N}}}
\newcommand{\R}{\ensuremath{\mathbb{R}}}

\newcommand{\E}{\ensuremath{\mathbb{E}}}
\renewcommand{\P}{\ensuremath{\mathbb{P}}}

\newcommand{\ind}[1]{\ensuremath{\mathbbm{1}_{\left\{#1\right\}}}}
\newcommand{\diff}{\mathop{}\mathopen{}\mathrm{d}}
\newcommand{\cal}[1]{\ensuremath{\mathcal{#1}}}
\newcommand\croc[1]{\left\langle #1\right\rangle}
\newcommand\steq[1]{\stackrel{\text{\rm #1.}}{=}}

\def\eps{\varepsilon}
\def\cadlag{c\`adl\`ag }

\renewcommand{\theenumi}{\alph{enumi}}

\newtheorem{proposition}{Proposition}
\newtheorem{definition}[proposition]{Definition}
\newtheorem{lemma}[proposition]{Lemma}
\newtheorem{theorem}[proposition]{Theorem}

\title[Scaling Methods for Stochastic CRN]{Scaling Methods for Stochastic Chemical Reaction Networks}

\date{\today}
\author[L. Laurence]{Lucie Laurence}
\email{Lucie.Laurence@unibe.ch}
\address[L. Laurence]{Institute of Mathematical Statistics and Actuarial Science, Department of Mathematics and Statistics, University of Bern, Alpeneggstrasse 22, 3012 Bern, Switzerland}
 \address[Ph. Robert]{INRIA Paris, 48, rue Barrault, CS 61534, 75647 Paris Cedex, France}
\author[Ph. Robert]{Philippe Robert}
\email{Philippe.Robert@inria.fr}
\urladdr{http://www-rocq.inria.fr/who/Philippe.Robert}

\begin{document}
\begin{abstract}
The  asymptotic properties of some Markov processes associated to  stochastic chemical reaction networks (CRNs) driven by  the kinetics of the law of mass action are analyzed. The scaling regime introduced in the paper assumes that the norm of the initial state is converging to infinity. The reaction rate constants are kept fixed.  The purpose of the paper is of showing, with simple examples, a scaling analysis in this context. The main difference with the scalings of the  literature is that it does not change  the graph structure of the CRN or its reaction rates. Several CRNs are investigated to illustrate the insight that can be gained on the qualitative properties of these networks. A detailed scaling analysis of a CRN with several interesting asymptotic properties, with a bi-modal behavior in particular,  is worked out in the last section. Additionally, with several examples, we also show that  a stability criterion due to  Filonov  for positive recurrence of Markov processes may simplify significantly the stability analysis of these networks. 
\end{abstract}
\maketitle

 \vspace{-5mm}

\bigskip

\hrule

\vspace{-3mm}

\tableofcontents

\vspace{-1cm}

\hrule

\bigskip

\section{Introduction}
This paper investigates the asymptotic properties of Markov processes $(X(t))$ associated to chemical reaction networks (CRNs).  The state space of these processes is a subset of $\N^n$, where  $n{\ge}1$ is the number of {\em chemical species}. A {\em chemical reaction} $y^-{\rightharpoonup} y^+$, for $y^-$, $y^+{\in}{\cal C}{\subset}\N^n$, where ${\cal C}$ is the set of {\em complexes},  is associated to a transition of  a Markov process on $\N^n$ of the form, for $x{=}(x_i){\in}\N^n$,
\[
x\longrightarrow x{+}\sum_{1}^n\left(y_i^+{-}y_i^-\right)e_i,
\]
where, for $i{\in}\{1,\ldots,n\}$,  $e_i$ is the $i$th unit vector of $\N^n$.  Such a reaction  $y^-{\rightharpoonup} y^+$ is represented by the couple $(y^-,y^+){\in}{\cal C}^2$ and the set of possible chemical reactions of the CRN is denoted by ${\cal R}$. 

The kinetics used classically for CRNs is the celebrated {\em law of mass action}, see~\citet{Guldberg}, \citet{lund1965guldberg} and~\citet{Voit2015}. 
This is expressed by the fact that the above transition has a rate proportional to 
\[
\frac{x!}{(x{-}y^-)!}\steq{def} \prod_{i=1}^n \frac{x_i!}{(x_i{-}y_i^-)!}= \prod_{i=1}^n x_i(x_i{-}1)\cdots (x_i{-}y_i^-{+}1),
\]
provided that $x_i{\ge}y_i^-$ holds for all $i{\in}\{1,\ldots,n\}$, it is $0$ otherwise. The proportionality constants, called the \emph{reaction rate constants}, are written as $(\kappa_r, r{=}(y^-,y^+){\in}{\cal R})$. 
The transition rates  exhibit:
\begin{enumerate}
\item A {\em polynomial dependence} on the state variable;
 \item {\em Boundary effects}: A chemical reaction requires a minimal number of copies of some chemical species to take place.	 The reaction occurs only when  $x_i{\ge}y_i^-$, holds for all $1{\le}i{\le}n$. 
\end{enumerate}
From the point of view of the mathematical analysis, these two properties are the main features of stochastic CRNs. They are at the origin of complex behaviors  with multi-timescales  and local equilibria. 

\subsubsection*{Deterministic CRNs}
It should be noted that boundary effects and multi-timescales do not play a role in the mathematical analysis of the historical models of {\em deterministic } CRNs. The time evolution of a deterministic CRN is described in terms of  the solution of an ODE with a polynomial  dependence on the state variable.  See~\citet{Voit2015}. 
For these networks there is a priori one timescale.  A  classical convergence result of scaled stochastic CRNs to a such deterministic CRN is achieved precisely by modifying  reaction rates so that all reaction rates have the same order of magnitude. See~\citet{Mozgunov}. 

A major result on deterministic CRNs,  the {\em deficiency zero Theorem}, due to Feinberg (1979), states that under some topological conditions, i.e. if the CRN is {\em weakly reversible and with zero deficiency}, see~\citet{Feinberg},   then there is exactly one fixed point for the dynamical system and this equilibrium is locally stable. An interesting feature of this class of CRNs is that this existence and uniqueness holds for any choice of the set of constants $(\kappa_{r}, r{\in}{\cal R})$ as long as they are all positive. See~\citet{Horn1972}, \citet{Horn1972b} and~\citet{Feinberg1972}.

\subsection{Scaling Pictures of Stochastic CRNs}
To investigate the qualitative properties of these networks, a possible approach is to use a scaling parameter, like the volume $N$  for example, and to derive convergence results for the sample paths $(X_N(t))$, with  a convenient scaling in time and space. \citet{BallKurtz} is one of the early works in this domain, where several specific examples are analyzed in this way. The reaction rate constants may be dependent on the scaling parameter, so that the CRN structure is also dependent on $N$. This can be seen as a generalization of  scaling approaches described in~\citet{Mozgunov}. One of the motivation of this work was of identifying the  corresponding scaling exponents of several biological systems, like the classical Michaelis-Menten reaction.  See~\citet{KangKurtz} for an extension of this approach. In this spirit,  convergence results have been obtained for several general classes of stochastic CRNs  in~\citet{Crudu} and~\citet{Enger}. 

\subsection*{Scaling with Large Initial States}
From the point of view of the positive recurrence properties of the associated Markov process, see Section~\ref{PosRecSec} below, it is quite natural to consider the norm of the initial state $x$ as a scaling parameter.  In this scaling regime the structure of the CRN, i.e. the topology and the set of reaction rate constants,  is fixed, it does not depend on the scaling parameter. Since the general structure of the CRN is fixed,  a convergence result describes how the CRN returns to a neighborhood of $0$ starting from a large initial state.

This approach  is developed for several examples of CRNs in our paper. If there are analogies with~\citet{BallKurtz} on the technical tools used in particular, there are significant differences nevertheless. Convergence results (in distribution) of corresponding scaled sample paths are not always possible. 
\begin{enumerate}
\item The renormalization in time and space may depend in an essential way on the location of the initial point, on the direction at infinity $x/\|x\|$, for $\|x\|$ large, for example. Convergence results may then significantly differ depending on the region of the unit ball on $\R_+^n$ considered. In~\citet{KangKurtz}  this does not happen, the scaling parameter $N$ determines completely the possible convergence result.
\item One may have to look beyond  some stopping time $\tau_{x}$, i.e.  consider a time interval $(\tau_x,{+}\infty)$ for such a convergence. This may be  due to the fact some coordinates do not have initially the ``right'' order of magnitude with respect to $\|x\|$. Recall that, a priori, we explore all values of $x/\|x\|$. There is one such situation in an example of~\citet{BallKurtz}.
\end{enumerate}
This is illustrated in Sections~\ref{Binary} and~\ref{CapSec} and also Sections~3 and~4 of~\citet{LR24}.

This  scaling approach  may give a quite precise picture of how the state of the CRN returns to $0$ and thus  may look appealing to obtain a proof of positive recurrence. Our experience is that it does not seem to be the case in general. This is mainly due to the fact that uniform estimates on the directions to infinity, on $x/\|x\|$,  are required for such a result. This is complicated when situations~(a) or (b) described above occur.  See Section~\ref{Binary} on binary CRNs.  Our point of view is that the proof of positive recurrence is better handled by a result due to Filonov, see below. 

The situation is in fact reminiscent of the studies of queueing networks where a proof of positive recurrence using the possible limits of the scaled process, the {\em fluid limits}, can be done for some examples, see~\citet{Stolyar} and~\citet{Dai}, but does not always apply, for basically the same reasons. See the interesting examples of~\citet{Bramson94} and~\citet{Rybko}. See also~\citet{Bramson}.  Challenging examples have to be handled in a different way in general. 

The interest of scaling results with the norm of the initial state lies, in our view,  in a precise description of the qualitative properties of the sample paths of a given CRN. For example, it gives a precise mathematical description of some phenomena occurring in stochastic models of CRNs, such that bi-modal behaviors, see Section~\ref{CapSec}, or DIT phenomenon for example, see~\citet{LR24}, \ldots

It should be noted that with this scaling the multiple timescales appear ``naturally''. They are dependent of the orders of magnitude of the coordinates of the state vector since the reaction rates $(\kappa_r)$ are constant. Finding the ``right'' orders of magnitude is by the way not an easy task. In~\citet{BallKurtz} and~\citet{KangKurtz}, they are determined by the scaling coefficients $N^\delta$ chosen for the reaction rates.

\subsection{Positive Recurrence Properties}\label{PosRecSec}
In a stochastic context, it is natural to investigate the positive recurrence properties of  Markov processes $(X(t))$ associated to CRNs. \citet{Anderson2010} has shown that if the topological structure of the stochastic CRN satisfies the assumptions of the deficiency zero theorem for deterministic CRNs, then the Markov process has an invariant probability distribution with a product form expression. In particular, it is positive recurrent.  This result  has been extended to other classes of CRNs in ~\citet{Cappelletti2016,Cappelletti2018} and~\citet{jia_jiang_li_2021}, \ldots

When a product-form formula for the invariant distribution does not hold, the positive recurrence property can be established by using  a Lyapunov function $f$ associated to the $Q$-matrix $Q$ of the Markov process. Such a function  satisfies a relation  of the type $Q(f)(x){\le} {-}\gamma$, for some $\gamma{>}0$ and for all $x$, ``large'', i.e. outside some finite subset of the state space. It amounts to the fact that  $(f(X(t))$ decreases in one step in average when the initial state is large. See Proposition~8.14 of~\citet{Robert}. We will refer to it as the {\em classical Lyapunov criterion}. This has been used in~\citet{Anderson2018}, \citet{Agazzi2018}, \citet{CapConf}, \ldots

A well-known problem is of finding such a function. In practice it is not difficult to figure out  a partition of  subsets of ``large'' states and to define an appropriate function on  each of these subsets of the partition where the function decreases in one step in this subset. The main problem is of gluing these functions: At the boundaries of the subset of the partition, a jump may change the subset of the partition which has an impact on the  value of $Q(f)(x)$. \citet{Agazzi2018} provides a good illustration of the difficulties of this approach.  See also~\citet{CapConf}.

In this paper we stress the importance of a,  somewhat underestimated/not well-known,  result of the literature, Filonov's Theorem~\cite{Filonov1989}.  In general, in our view, it simplifies the proof of positive recurrence of CRNs.  Instead of looking at the next jump, as in the classical Lyapunov criterion, one may consider a random number of steps, by looking the process at the instant of a stopping time for example. With this method,  gluing problems of the one-step criterion may be avoided.  See Theorem~\ref{ErgCrit} for a precise formulation. With this approach, the partitioning  of the state space is done in fact on the initial states and not on the arrival states of a transition as for the classical Lyapunov criterion. We give several examples of its use in practice: in Section~\ref{SecAgaz} in particular, the positive recurrence analysis of the CRN considered in~\citet{Agazzi2018} is significantly simplified with this approach. 

It should be mentioned that an interesting notion of {\em tier structure} has been introduced in~\citet{Anderson2020}. Combined with the classical Lyapunov criterion, it provides another approach to the proof of positive recurrence of some stochastic CRNs, for which dominant reactions dissipate the chosen energy. It will not be discussed in this paper. 

\subsection{Overview of the Paper}
The paper is in the same spirit as~\citet{BallKurtz}, we consider essentially examples.  Stochastic CRNs have in general a  quite complex behavior. It is not well understood how to handle multiple timescales outside of the classical stochastic averaging framework of~\citet{Kurtz1992}, i.e.  when there are more than two timescales. And similarly for the impact of boundary effects. See~\citet{LR24-2,LR24}.

It seems to us that, for the moment,  considering interesting  examples is perhaps a possible way of making progress on these networks, to develop methods and results and also to identify  the important phenomena associated to CRNs.  In this paper, we have chosen several examples to illustrate several aspects of this approach. 
\begin{enumerate}
  \item The insight on their qualitative properties that can be obtained by a scaling analysis with the norm of the initial state;
\item The benefit of considering Filonov's Criterion to analyze their positive recurrence properties.
\end{enumerate}
The formal definitions and notations are introduced in Section~\ref{00ModelSec}.  
\subsubsection{Binary Stochastic CRNs}
Section~\ref{Binary} is devoted to the analysis of some binary CRNs, that are chemical reaction networks whose complexes have at most two molecules. A simple example of {\em triangular} network is considered with three complexes and the sink/source $\emptyset$. The proofs in this section are essentially elementary, the main motivation is to show  how Filonov's result can be used. Even in this simple setting,  the formulation of a convenient convergence result is not straightforward. There are cases with three regimes corresponding to different timescales and different functional limit theorems. This case also provides an example of how technical results on hitting times for the $M/M/\infty$ queue can be used to establish convergence results for CRNs. See Section~\ref{MMIsec}.  The  $M/M/\infty$ queue  in fact the basic CRN,
\[
\emptyset \mathrel{\mathop{\xrightleftharpoons[\mu]{\lambda}}} S_1.
\]
The role of these technical results in the study of CRNs is, up to now, perhaps not widely realized. See~\citet{LR24}.

\subsubsection{Agazzi and Mattingly's CRNs}
Section~\ref{SecAgaz} analyzes an interesting CRN proposed in~\citet{Agazzi2018}.  The purpose of this reference is of showing that with a small modification of the graph structure of a CRN, its associated Markov process can be either positive recurrent, null recurrent, or transient. The main technical part of this reference is  essentially devoted  to the construction of a Lyapunov function satisfying  the classical Lyapunov criterion. We show that Filonov's criterion can be in fact used with a simple function to prove the positive recurrence. Additionally, a scaling picture of the time evolution of this CRN is obtained.

\subsubsection{A CRN With A Bi-Modal Behavior}
In Section~\ref{CapSec}, a detailed analysis of the following CRN, for  $p{\ge}2$, 
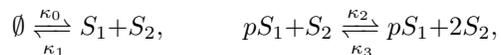
\[
\emptyset \mathrel{\mathop{\xrightleftharpoons[\kappa_{1}]{\kappa_{0}}}}  S_1{+}S_2, \hspace{1cm} pS_1{+}S_2 \mathrel{\mathop{\xrightleftharpoons[\kappa_{3}]{\kappa_{2}}}}pS_1{+}2S_2,
\]
is achieved. When $p{=}2$, this is an important  example introduced in~\citet{CapConf} for the stability analysis of a large class of CRNs with two chemical species. 

This CRNs has a boundary effect in the sense that the second reaction does not occur when the first coordinate is less than $p$. The corresponding scaling results are deeply impacted by this discontinuity of the dynamic. This creates a natural bi-modal behavior. This situation does not seem to fit in the framework of  the general results of ~\citet{Crudu} and~\citet{Enger} or by~\citet{KangKurtz}. One of the reasons is that one of our scaling results involves an {\em explosive} Markov process on $(0,1]$ with a multiplicative structure, whose infinitesimal generator ${\cal A}$ is defined below. To the best of our knowledge this is  quite original in the literature of stochastic CRNs.  A related phenomenon has been investigated in~\citet{LR24}, with a non-explosive Markov  process  but also with a multiplicative component. We do believe that this type of property, which has not been thoroughly investigated in general, holds for a quite large class of CRNs.

The analysis of this apparently simple CRN  has to be handled with care.  We first show how Filonov's Criterion can be used for positive recurrence and then  a scaling analysis is achieved to get  interesting insights for the time evolution of this CRN.  It also gives an interesting example of the use of  estimates of stopping times, time change arguments, \ldots to derive scaling results for these models.

The  bi-modal  property is exhibited via a scaling analysis of this CRN  for  two  classes of initial states. The corresponding limiting results are: 
\begin{enumerate}
\item\label{it1} For an initial state of the form $(N,b)$, with $b{\in}\N$ fixed.\\
Theorem~\ref{TheoCapH} shows that the convergence in distribution of processes
  \[
  \lim_{N\to+\infty}\left(\frac{X_1(Nt)}{N},t{<}t_\infty\right)=\left(1{-}\frac{t}{t_\infty},t{<}t_\infty\right),
  \]
  holds with
  \[
 t_\infty{=}{1}\left/{\kappa_0\left(e^{\kappa_2/\kappa_3}{-}1\right)}\right.. 
  \]
\item\label{it2} If the initial state is of the form $(a,N)$, $a{<}p$.\\
For the convergence in distribution of its occupation measure, see Definition~\ref{defOcc}, the relation
  \[
  \lim_{N\to+\infty}\left(\frac{X_2(N^{p-1}t)}{N}\right)=(V(t))
  \]
holds, where $(V(t))$ is  a Markov process on $(0,1]$ with a multiplicative structure, its infinitesimal generator ${\cal A}$ is given by 
\[
{\cal A}(f)(x)=\frac{r_1}{x^{p-1}}\int_0^1\left(f\left(xu^{\delta_1}\right){-}f(x)\right)\diff u, \quad x{\in}(0,1],
  \]
  for any Borelian function $f$ on $(0,1)$, with $\delta_1{=}{\kappa_3(p{-}1)!}/{\kappa_1}$. This process is explosive and is almost surely converging to $0$. 
\end{enumerate}
In this example, to decrease the norm of the process, one has to use the timescale $(Nt)$ in~\eqref{it1} and $(N^{p-1}t)$ in~\eqref{it2} and  the decay in~\eqref{it1} is  linear with respect to time. This is in contrast with the examples of Sections~\ref{Binary} and~\ref{SecAgaz} where the ``right'' timescale to see the energy decrease is of the form $(t/N^\beta)$ with $\beta{\in}\{0,1/2,1\}$.  Note also that the limit of the first order of~\eqref{it2} is a {\em random} process, instead of a classical deterministic function solution of an ODE as it is usually the case. 
\section{Mathematical Models of CRNs}\label{00ModelSec}

\subsection{General Definitions for CRNs}
We now give the formal definitions for chemical reaction networks. 
\begin{definition}
A {\em chemical reaction network} (CRN) with $n$ {\em chemical species}, $n{\ge}1$, is defined by a triple $({\cal S},{\cal C},{\cal R})$,
\begin{itemize}
\item ${\cal S}{=}\{1,\ldots,n\}$ is the set of chemical  species;
\item ${\cal C}$, the set of {\em complexes}, is a finite subset of $\N^n$;
\item ${\cal R}$, the set of {\em chemical reactions}, is  a subset of ${\cal C}^2$.
\end{itemize}
\end{definition}
A chemical species $j{\in}{\cal S}$ is also represented as $S_j$. A complex $y{\in}{\cal C}$, $y{=}(y_j)$ is composed of $y_j$ \emph{molecules} of species $j{\in}{\cal S}$, its \emph{size} is $\|y\|{=}y_1{+}\cdots{+}y_n$. It is also described as
\[
y=\sum_{j=1}^n y_jS_j.
\]
The state of the CRN is given by a vector  $x{=}(x_i,1{\le}i{\le}n){\in}\N^n$, for $1{\le}i{\le}n$, $x_i$ is the number of copies of chemical species $S_i$.
A chemical reaction $r{=}(y_r^-,y_r^+){\in}{\cal R}$ corresponds to the transition of state, for $x{=}(x_i)$, 
\begin{equation}\label{eqCRN}
x\longrightarrow x{+}y_r^+{-}y_r^-=\left(x_i{+}y_{r,i}^+{-}y_{r,i}^-, 1{\le}i{\le}n\right)
\end{equation}
provided that $y_{r,i}^-{\le}x_i$ holds for  $1{\le}i{\le}n$, i.e. there are at least $y_{r,i}^-$ copies of chemical species of type $i$, for all $i{\in}{\cal S}$, otherwise the reaction cannot happen. Such a chemical reaction is classically represented as 
\[
\sum_{i=1}^n y_{r,i}^-S_i \rightharpoonup \sum_{i=1}^n y_{r,i}^+S_i,
\]
The notation $\emptyset$  refers to the complex associated to the null vector of $\N^n$,  $\emptyset{=}(0)$. For $y{=}(y_i) {\in}{\cal C}$, a chemical reaction of the type $(\emptyset, y)$  represents  an external source creating $y_i$ copies of species $i$, for $i{=}1$,\ldots, $n$.  A chemical reaction of the type $(y,\emptyset)$ consists in removing $y_i$ copies of species $i$, for $i{=}1$,\ldots, $n$, provided that there are sufficiently many copies of each  species.

\subsection{Law of Mass Action}\label{MarkovSec}
A stochastic model of a CRN  is represented by  a continuous time Markov jump process $(X(t)){=}(X_i(t),i{=}1,\ldots,n)$ with values in $\N^n$. The dynamical behavior of a CRN, i.e. the time evolution of the number of copies of each of the $n$ chemical species is governed by {\em the law of mass action}.  See~\citet{Voit2015}, \citet{lund1965guldberg} for surveys on the law of mass action.

For these kinetics, the associated $Q$-matrix of $(X(t))$ is defined so that, for $x{\in}\N^n$ and $r{=}(y_r^-, y_r^+){\in}{\cal R}$, the transition $x{\to}x{+}y_r^+{-}y_r^-$ occurs at rate 
\begin{equation}\label{00Qmat}
\kappa_rx^{(y_r^-)},
\end{equation}
where  $\kappa{=}(\kappa_r,r{\in}{\cal R})$ is a vector of non-negative numbers, for $r{\in}{\cal R}$, $\kappa_r$ is the reaction rate constant of $r$ and,
for $z{=}(z_i){\in}\N^n$ and $y{=}(y_i){\in}{\cal C}$,
\begin{equation}\label{Nota3}
z!\steq{def}\prod_{i=1}^n z_i!,\quad
z^{(y)}\steq{def}\frac{z!}{(z{-}y)!}=\prod_{i=1}^n \frac{z_i!}{(z_i{-}y_i)!},
\end{equation}
with the convention that $z^{(y)}{=}0$, if there exists some $i_0{\in}{\cal S}$ such that $y_{i_0}{>}z_{i_0}$.

The functional operator $ {\cal Q}(f)$ associated to this $Q$-matrix is defined by, for $x{\in}\N^n$, 
\begin{equation}\label{Qmat2}
 {\cal Q}(f)(x) = \sum_{r{\in}{\cal R}}\kappa_rx^{(y_r^-)}\left(f\left(x{+}y_r^+{-}y_r^-\right){-}f(x)\right),
\end{equation}
 for any function $f$ with finite support on $\N^n$. 

\subsection{An Important CRN: The $\mathbf{M/M/\infty}$ queue}\label{MMIsec}
This is a simple CRN with an external input and one chemical species,
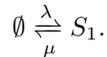
\[
\emptyset \mathrel{\mathop{\xrightleftharpoons[\mu]{\lambda}}} S_1.
\]
The ${M/M/\infty}$ queue with input parameter $\lambda{\ge}0$ and output parameter $\mu{>}0$ is a  Markov process $(L(t))$  on $\N$ with transition rates
\[
x\longrightarrow
\begin{cases}
x{+}1&   \lambda \\
x{-}1&   \mu x.
\end{cases}
\]
The invariant distribution of $(L(t))$ is Poisson with parameter $\rho{=}\lambda/\mu$.

This fundamental process can be seen as a kind of discrete Ornstein-Uhlenbeck process. It has a long history, it has been used in some early mathematical models of telephone networks at the beginning of the twentieth century, see~\citet{Erlang09},  also in stochastic models of natural radioactivity in the 1950's, see~\citet{Hammersley} and it is the basic process of mathematical models of communication networks analyzed in the 1970's, see~\citet{Kelly}.  See Chapter~6 of~\citet{Robert}.

Technical results on this stochastic process turn out to be useful to investigate the scaling properties of some CRNs and, as we will see,  in the construction of couplings used in our proofs.

\subsection{Filonov's Stability  Criterion}\label{SecStab}
In this section we formulate a  criterion, due to Filonov~\cite{Filonov1989}, of positive recurrence for continuous time Markov processes associated to CRNs. It is an extension of the classical Lyapunov criterion, see Proposition~8.14 of~\citet{Robert}.  In our experience, it turns out to be very useful in the context of CRNs. See Theorem~8.6 of~\cite{Robert}.
\begin{definition}
 An energy function $f$ on ${\cal E}_0$ is a non-negative function such that, for all $K{>}0$, the  set $\{x{\in}{\cal E}_0: f(x){\le}K\}$ is finite. 
\end{definition}

\begin{theorem}[Filonov]\label{ErgCrit}
Let $(X(t))$ be an irreducible Markov process on ${\cal E}_0{\subset}\N^n$ associated to a CRN network with $Q$-matrix~\eqref{00Qmat}. If there exist
  \begin{enumerate}
  \item an integrable stopping time $\tau$ and $\eta{>}0$,  such that  $\tau{\ge}t_1{\wedge}\eta$,\\ 
for a constant $\eta{>}0$ and $t_1$ is the first jump of $(X(t))$;
  \item an energy function  $f$ on ${\cal E}_0$ and constants $K$ and $\gamma{>}0$ such that  the relation
  \begin{equation}\label{LyapCond}
    \E_x\left(f(X(\tau))\right){-}f(x)\le {-}\gamma \E_x(\tau),
  \end{equation}
  holds   for all $x{{\in}}{\cal E}_0$ such that $f(x){\ge}K$, 
  \end{enumerate}
  then  $(X(t))$ is a positive recurrent  Markov process. 
\end{theorem}
A function $f$ satisfying Condition~\eqref{LyapCond} is usually  referred to as a {\em Lyapunov function}. 

\section{Binary CRN Networks}\label{Binary}
In this section, we investigate simple examples of CRNs with complexes whose size is at most  $2$. 
\begin{definition}
A CRN network $({\cal S},{\cal C},{\cal R})$ with $n$ chemical species is {\em binary} if any complex $y{\in}{\cal C}$ is composed of at most two molecules, i.e.~$\|y\|{\leq}2$.
\end{definition}
The set of complexes can be represented as  ${\cal C}{=}J_0{\cup}J_1{\cup}J_2$, where, for $i{\in}\{0,1,2\}$, the subset $J_i$ is the set of complexes with $i$ chemical species, $J_i$ can be empty. An element  $y{\in}J_1$ is represented as $y{=}s_y$ for some $s_y{\in}{\cal S}$. Similarly,  $y{\in}J_2$ is written as $y{=}(s_y^1,s_y^2)$, with $s_y^1, s_y^2{\in}{\cal S}$.

\begin{proposition}\label{ODEBinProp}
  If $(X_N(t))$ is a sequence of Markov processes associated to a binary CRN with $n$ chemical species whose sequence of initial states $(x_N)$  is such that $(\|x_N\|)$ is converging to infinity and 
  \[
  \lim_{N\to+\infty}\frac{x_N}{\|x_N\|}=\ell_0{\in}\R_+^n,
  \]
then the family of random variables
\[
\left(\overline{X}_N(t)\right)\steq{def}\left(\frac{1}{\|x_N\|}X_N(t/\|x_N\|)\right),
\]
  converges in distribution to the solution $(\ell(t))$ of the ODE
\begin{equation}\label{ODEBin}
 \dot{\ell}(t)=\sum_{\substack{r{=}(y_r^-,y_r^+)\in {\cal R}\\y_r^-{\in}J_2}} \kappa_{r}(y_r^+{-}y_r^-)\ell_{s_{y_r^-}^1}(t)\ell_{s_{y_r^-}^2}(t),
\end{equation}
with initial state $\ell(0){=}\ell_0$.
\end{proposition}
\begin{proof}
  The arguments of the proof use standard stochastic calculus arguments,  they are omitted. 
\end{proof}
\bigskip

It should be noted that the timescale $(t/\|x\|)$  and the space scale $1/\|x\|$ are valid for all binary CRNs from the point of view of tightness properties. It does not mean that they are the only ones, or the most meaningful.  The timescale $(t/\|x\|)$ is well-suited when there are complexes of size two and when the associated chemical species are all in ``large'' number, of the order of $\|x\|$. Otherwise, it may be too slow to change the state of the CRN,  so that a faster timescale has to be used.  As it will be seen, depending on the type of initial state, it may happen that the timescales $(t/\sqrt{\|x\|})$ or  $(t)$  and the space scales $1/\sqrt{\|x\|}$ or $1$ are appropriate for the analysis of the asymptotic behavior of the time evolution of the CRN. The following simple example illustrates these considerations.

\subsection*{A Triangular Binary Network}\label{Ex2Sec}
The binary CRN studied, with two species, is represented by the following graph of reactions: 
\begin{figure}[ht]
\begin{tikzpicture}[->,node distance=1.5cm]
  \node (A) [below] {$S_2$};
  \node (B) [right of=A] {$S_1{+}S_2$};
  \node (C) [below of=B] {$S_1$};
  \node(D)[below of=A]{$\emptyset$};
     \draw[-left to]  (A) -- node[rotate=0,above] {$\kappa_{2}$}   (B);
     \draw[-left to]  (D) -- node[rotate=0,left] {$\kappa_{02}$}   (A);
     \draw[-left to]  (D) -- node[rotate=0,below] {$\kappa_{01}$}   (C);
     \draw[-left to] (B) -- node[rotate=0,right] {$\kappa_{12}$} (C);
     \draw[-left to] (C) -- node[rotate=0,below] {$\kappa_{1}$} (A);
\end{tikzpicture}
\caption{Triangle CRN}\label{FigC1}
\end{figure}
The purpose of this section is essentially pedagogical, to show, in a simple setting, how the Filonov's criterion can be used in practice by using scaling results and other ad-hoc arguments and also  how a possible result of  convergence in distribution  of scaled sample paths may depend (in a simple way here)  on the direction to infinity $x/\|x\|$.

We consider a sequence $(x_N)$ of initial states such that
\begin{equation}\label{InitTRG}
\lim_{N\to+\infty}\left(\frac{x_1^N}{\|x_N\|},\frac{x_2^N}{\|x_N\|}\right)=(\alpha_1,1{-}\alpha_1),
\end{equation}
with $\alpha_1{\in}[0,1]$ and the associated  Markov process  with initial state $x_N$  is denoted by $(X_N(t)){=}(X^N_1(t),X^N_2(t))$. 

The scalings consider three types of regions of $\N^2$ for the initial state: when the order of magnitude of the two  coordinates are respectively of the order of $(N,N)$,  $(O(\sqrt{N}),N)$, or $(N,O(1))$. It is shown that starting from a ``large'' state,  three timescales play a role depending on the asymptotic behavior of the initial state:
    \begin{enumerate}
    \item $t\mapsto t/N $, when both components of the initial state are of the order of $N$, i.e. when $0{<}\alpha_1{<}1$;
    \item $t\mapsto t/\sqrt{N}$, when $\alpha_1{=}0$ and $x^N_1$ is at most of the order of $\sqrt{N}$. 
    \item $t\mapsto t$, when $\alpha_1{=}1$ and $x_2^N$ is bounded by some constant $K$. 
    \end{enumerate}
 The boundary effects mentioned in Section~\ref{00ModelSec} play a role in case c), the second coordinate remains in the neighborhood of the origin essentially.

For each of the three regimes, the scaled norm of the state is decreasing to $0$. The limit results show additionally that the orders of magnitude in $N$ of both coordinates do not change. In other words the space scale is natural and not the consequence of a specific choice of the timescale. The following proposition gives a formal statement of these assertions. 

\begin{proposition}[Scaling Analysis]\label{ConvT2Prop}
Under the assumptions~\eqref{InitTRG}, if $(X^N_1(t),X^N_2(t))$ is the Markov process associated to the CRN  of Figure~\ref{FigC1}, we have 
\begin{enumerate}
\item if $\alpha_1{>}0$, then  for the convergence in distribution,
\begin{equation}\label{CvT2S1}
  \lim_{N\to+\infty} \left(\frac{X^N_1(t/N)}{N},\frac{X^N_2(t/N)}{N}\right)=(x_{a,1}(t),x_{a,2}(t)),
\end{equation}
where   $(x_{a,1}(t),x_{a,2}(t)){=}(\alpha_1,(1-\alpha_1)\exp(-\kappa_{12}\alpha_1 t))$.
\item If $\alpha_1{=}0$ and
\[
  \lim_{N\to+\infty} \frac{x^N_1}{\sqrt{N}}=\beta \in \R_+,
\]
then, for the convergence in distribution
\begin{equation}\label{CvT2S2}
  \lim_{N\to+\infty} \left(\frac{X^N_1(t/\sqrt{N})}{\sqrt{N}},\frac{X^N_2(t/\sqrt{N})}{N}\right)=(x_{b,1}(t),x_{b,2}(t)),
 \end{equation}
where $(x_{b,1}(t),x_{b,2}(t))$ is the solution of the ODE
\begin{equation}\label{OdeT2S2}
  \dot{x}_{b,1}(t)=\kappa_2 x_{b,2}(t),\quad   \dot{x}_{b,2}(t)={-}\kappa_{12}x_{b,1}(t)x_{b,2}(t),
\end{equation}
with $(x_{b,1}(0),x_{b,2}(0)){=}(\beta,1)$. 
\item If the initial state is  $x_N{=}(N,k)$, for $k{\in}\N$, then, for the convergence in distribution,
\begin{equation}\label{CvT2S3}    
\lim_{N\to+\infty} \left(\frac{X^N_1(t)}{N}\right)= (x_{c,1}(t))\steq{def}\left(e^{-\kappa_1 t}\right).
\end{equation}
\end{enumerate}
\end{proposition}
\begin{proof}
We give a quick proof of  the convergence~\eqref{CvT2S3} to illustrate the role of coupling methods with $M/M/\infty$ models to handle some technical difficulties.
The  proofs of the convergences~\eqref{CvT2S2} and~\eqref{OdeT2S2} use  similar ingredients and will be skipped for this reason.

The initial state is $x^N{=}(x^N_1, k)$ for some $k\in \N$, and $x_1^N$ such that
\[
  \lim_{N\to+\infty} \frac{x^N_1}{N}=1.
  \]
 For simplicity we will consider the CRN without external arrivals, i.e. $\kappa_{01}{=}\kappa_{02}{=}0$, the proof for the general case is similar. Provided that the process $(X_1^N(t))$ is of the order of $N$, the external arrivals of chemical species $1$ is  negligible for the first coordinate over a finite time interval. Similarly, the creation of chemical species $2$ is essentially due to  chemical species $1$ (at a rate of the order of $N$ much large than $\kappa_{02}$).
  
The associated Markov process $(X^N(t)){=}(X^N_1(t),X^N_2(t))$ can be represented as a solution of the following SDEs,
\begin{equation}\label{EqT1}
\begin{cases}
  \diff X^N_1(t)&={\cal P}_{2}((0,\kappa_{2}X^N_2(t{-})),\diff t){-}{\cal P}_1((0,\kappa_1X^N_1(t{-})),\diff t)\\
  \diff X^N_2(t)& ={\cal P}_1((0,\kappa_1X^N_1(t{-})),\diff t){-}{\cal P}_{12}((0,\kappa_{12}X^N_1(t{-})X^N_2(t{-})),\diff t),
\end{cases}
\end{equation}
with $X^N(0){=}(x_1^N,x_2^N)$. The independent Poisson processes $({\cal P}_i,i{=}1,2)$ have the intensity measure $\diff x\diff y$. See Section~\ref{SDESec} of the Appendix for the main definitions and notations of this setting. 

Define, for $i=1,2$,  
\[
\overline{X}^N_i(t)\steq{def}\frac{X^N_i(t)}{N}.
\]

We fix $T{>}0$. We have, for $t{\ge}0$,
\[
        X^N_1(t)\ge \sum_{i=1}^{X^N_1(0)}\ind{E_{\kappa_1,i}\ge  t},
\]
where $(E_{\kappa_1,i})$ is an i.i.d. sequence of exponential random variables with parameter $\kappa_1$, hence, for $\eps{>}0$,  there exists $\eta{>}0$ such that
\[
\P\left({\cal E}_N\right)\ge 1{-}\eps, \text{ with } {\cal E}_N=\left\{\eta\le \inf_{t\le T}\overline{X}^N_1(t) \right\}. 
\]
Define
\[
\tau_N=\inf\{t\ge 0: \overline{X}^N_1(t)\ge 2\},
\]
then a simple coupling shows that on the event ${\cal E_N}{\cap}\{\tau_N{\ge}T\}$, we have $X^N_1(t){\le}Y^N_1(t)$
and $X^N_2(t){\le}Y^N_2(t)$, for all $t{\le}T$, where $(Y^N_1(t),Y^N_2(t))$ is the solution of the SDE,
\begin{equation}\label{EqT2}
\begin{cases}
  \diff Y^N_1(t)&={\cal P}_{2}((0,\kappa_{2}Y^N_2(t{-})),\diff t){-}{\cal P}_1((0,\kappa_1Y^N_1(t{-})),\diff t)\\
  \diff Y^N_2(t)& ={\cal P}_1((0,2\kappa_1 N),\diff t){-}{\cal P}_{12}((0,\kappa_{12}\eta N Y^N_2(t{-})),\diff t),
\end{cases}
\end{equation}
with initial point $Y_N(0){=}X_N(0)$. 

The process $(Y^N_2(t))$ has the same distribution as $(L(Nt))$ where $(L(t))$ is an $M/M/\infty$ queue with input parameter $2\kappa_1$ and output parameter $\kappa_{12}\eta$.
\[
\P\left(\sup_{s{\le} T} \frac{Y^N_2(s)}{\sqrt{N}}\ge \eps\right)=\P\left(H_{\lfloor \eps \sqrt{N} \rfloor}\le N T\right),
\]
where,  $a{\in}\N$, $H_a$ is the hitting time of $a$ by $(L(t))$. Proposition~6.10 of~\citet{Robert} shows that if $\rho{=}2\kappa_1/(\kappa_{12}\eta)$, then the sequence of random variables $(\rho^aH_a/(a{-}1)!)$ converges in distribution to an exponentially distributed random variable as $a$ goes to  infinity.  Consequently, we obtain that $(Y^N_2(t)/N,0{\le}t{\le}T)$ converges in distribution to $0$.

It is then easy to show that $(Y^N_1(t)/N,0{\le}t{\le}T)$ converges in distribution to $(\exp({-}\kappa_1 t))$, from the relation $X^N_1(t){\le}Y^N_1(t)$ on the event ${\cal E_N}{\cap}\{\tau_N{\ge}T\}$, we conclude that $(P(\tau_N{\ge}T))$ converges to $0$ and therefore the convergence of  $(X^N_1(t)/N)$ to $(\exp({-}\kappa_1 t))$. Convergence~\eqref{CvT2S3} is proved. 
\end{proof}

\begin{proposition}[Stability Properties]
The Markov process  $(X(t)){=}X_1(t),X_2(t))$ associated to the CRN  of Figure~\ref{FigC1} is positive recurrent. 
\end{proposition}
\begin{proof}
In view of Theorem~\ref{ErgCrit}, we have to define a stopping $\tau$ time depending on the initial state. The norm $f_0{\steq{def}}\|{\cdot}\|$ is used as the energy function. As before we also ignore the external arrivals since, with high probability,  the stopping time chosen is (much) smaller than the instant of the first external arrival.
There are three cases. 
\renewcommand{\theenumi}{\arabic{enumi}}
\begin{enumerate}
\item When the initial state $X(0){=}x{=}(x_1,x_2)$ is such that
	\[
		x_1\geq C_0\steq{def} \frac{\kappa_2{+}1}{\kappa_{12}} \quad \text{and }\quad x_2\ge 1. 
	\]
We take $\tau{=}t_1{\wedge}1$, where $t_1$ is the instant of the first jump of $(X(t))$. We write $\cal{Q}$ the infinitesimal generator associated to $(X(t))$, see Relation~\eqref{Qmat2}. 
We have 
\[
	\cal{Q}(f_0)(x)=x_2(\kappa_2-\kappa_{12}x_1), 
\]
and therefore, by using Relation~\eqref{EqT1}, 
\begin{multline*}
\E_x\left(\|X(t_1{\wedge}1)\|-\|x\|\right)=\E_x\left(\int_0^{t_1{\wedge}1}\cal{Q}(f_0)(X(s))\diff s \right)
\\=x_2(\kappa_2-\kappa_{12}x_1)\E_x\left(t_1{\wedge}1\right)
\le{-}\E_x\left(t_1{\wedge}1\right).
\end{multline*}
\item The initial state is $x{=}(x_1, 0)$.\\
Let, for $k_0{\in}\N$, $\tau_{k_0}$ be the instant when the $(k_0{+}1)$th element $S_1$ is transformed into an $S_2$.
The norm of the process decreases when some of these molecules of $S_2$ disappear with reaction $S_1{+}S_2{\rightharpoonup}S_1$. The probability for a molecule of $S_2$ to be removed before a new transformation of $S_1$ into $S_2$ is lower bounded by $p{=}\kappa_{12}/(\kappa_1{+}\kappa_{12})$, therefore in average there are more than $pk_0$ molecules of $S_2$ killed before $\tau_{k_0}$. 
	The reaction $S_2{\rightharpoonup} S_1{+}S_2$ could also create some molecules during the time interval $[0,\tau_{k_0}]$. The rate of this reaction is bounded by $\kappa_2k_0$,  it is not difficult to show that there exists a constant $C_0$ such that the relations
  \[
    \E(\tau_{k_0}){\le} \frac{k_0}{\kappa_1(\|x\|{-}k_0)}\quad \text{and}\quad \E\left(\|X(\tau_{k_0})\|\right)\le \|x\|{-}k_0p{+}k_0\frac{C_0}{\|x\|{-}k_0}\, , 
  \]
hold. It is not difficult to see that the above relations are still valid, up to a change of constant, when $\tau_{k_0}$ is replaced by  $\tau_{k_0}{\wedge}1$. 
\item The initial state is $x{=}(x_1, x_2)$ with, $x_1{\le}C_0$.
  Define
  \[
  \tau_0=\inf\{t>0: X_1(t){\ge} C_0\},
  \]
  up to time $\tau_0$, the network is essentially similar to $S_2{\rightharpoonup}S_1{+}S_2{\rightharpoonup}S_1$. 
  It is not difficult to show that
  \[
  \lim_{x_2\to+\infty}x_2\E_{x}(\tau_0)=\frac{C_0}{\kappa_2} \text{ and }\lim_{x_2\to+\infty}\frac{\E_x(X_2(\tau_0))}{x_2}=1. 
  \]
At time $\tau_0$, the  state of the network network  is as in case~(1), it is enough to take $\tau$ as $\tau_0{+}\widetilde{\tau}$, where $\widetilde{\tau}$ is the variable $t_1{\wedge}1$ of~(1) associated to the process $(X(\tau_0{+}\cdot))$. More formally, if $(\theta_t)$ is the time-shift operator for the Markov process $(X(t))$, it is expressed as $\tau{=}\tau_0{+}t_1{\circ}\theta_{\tau_0}{\wedge}1$. See~\citet{Sharpe}. 
\end{enumerate}

\end{proof}
\section{Agazzi and Mattingly's CRN}\label{SecAgaz}
In this section, we study the chemical reaction network introduced by~\citet{Agazzi2018},
\begin{equation}\label{AgMattCRN}
 \emptyset \mathrel{\mathop{\xrightharpoonup{\kappa_1}}} S_1{+}S_2, \quad S_2\mathrel{\mathop{\xrightharpoonup{\kappa_2}}}\emptyset,\quad
 pS_1{+}q S_2\mathrel{\mathop{\xrightharpoonup{\kappa_3}}} (q{+}1)S_2\mathrel{\mathop{\xrightharpoonup{\kappa_4}}} qS_2,
\end{equation}
for $p$, $q{\in}\N$, $p{>}2$ and $q{\geq}2$. In~\citet{Agazzi2018}, the constants considered are $p{=}5$ and $q{=}2$. 

This reference shows that with a small modification of the topology of this CRN, its associated Markov process can be positive recurrent, null recurrent, or transient. The main technical part of the paper is devoted essentially to the construction of an energy function satisfying the classical Lyapunov criterion. The energy function  is defined in terms of polynomial functions in $x_1$ and $x_2$, on a partition of subsets of $\N^2$. The main technical difficulty is of gluing these functions in order to have a global function $f$ satisfying the classical Lyapunov criterion.  Note that there are also  interesting null-recurrence and transience properties in this reference.

In this section, Proposition~\ref{ErgCrit} gives a simple proof of the positive recurrence result of~\citet{Agazzi2018} by taking the norm as an energy function and a convenient stopping time depending on the initial state.

The continuous time Markov jump process $(X(t)){=}(X_1(t), X_2(t))$ associated to CRN~\eqref{AgMattCRN} has a $Q$-matrix given by, for $x{\in}\N^2$,
\[
x\longrightarrow x{+}\begin{cases}
  e_1{+}e_2& \kappa_1,\\
  {-}pe_1{+}e_2& \kappa_3 x_1^{(p)}x^{(q)}_2,\\
  {-}e_2& \kappa_2x_2{+}\kappa_4x^{(q+1)}_2,
\end{cases}
\]
where $e_1$, $e_2$ are the unit vectors of $\N^2$. This process is clearly irreducible on $\N^2$, and non explosive since $\emptyset {\rightharpoonup} S_1{+}S_2$ is the only reaction increasing the total number of molecules.  The fact that only this reaction increases the norm of the state  suggests that the proof  of the positive recurrence  should not be an issue.

\begin{proposition}\label{AMProp}
If $p{>}2$ and $q{\ge}2$, then the Markov process associated to the CRN~\eqref{AgMattCRN} is positive recurrent.
\end{proposition}
\begin{proof}
Theorem~\ref{ErgCrit} is used with a simple energy function, the norm $\|x\|{=}x_1{+}x_2$  of the  state $x{=}(x_{1},x_{2}){\in}\N^2$. If the norm of the initial state is large enough, then the expected value of the norm of the process taken at a convenient stopping time  will be smaller, so that Condition~\eqref{LyapCond} of Theorem~\ref{ErgCrit} holds. 

\subsection*{Step~1}
As before, for $n{\ge}1$, $t_n$ denotes the instant of the $n$th jump. 
\[
\E_x\left(\|X(t_1)\|{-}\|x\|\right) = \left(2\kappa_1{-}\kappa_2 x_{2}{-}(p{-}1)\kappa_3 x_{1}^{(p)}x_{2}^{(q)}{-}\kappa_4x_{2}^{(q+1)}\right)\E_x\left[t_1\right],
\]
and, clearly, $\E_x(t_1){\le}{1}/{\kappa_1}$. 

If either $x_{2}{\ge}K_1{=}1{+}2\kappa_1/\kappa_2$ or $q{\le}x_{2}{<}K_1$ and $x_{1}{\ge}K_2{=}1{+}2\kappa_1/((p{-}1)\kappa_3q!)$, then
\[
\E_x\left(\|X(t_1)\|{-}\|x\|\right) \le {-}\gamma \E_x(t_1),
\]
for some $\gamma{>}0$.  Condition~\eqref{LyapCond} holds for this set of initial states. 

\subsection*{Step~2}
Now we consider initial states of the form $x_N{=}(N, b)$ with $b{<}q$ and $N$ large. 
The third and fourth reactions cannot occur until the instant
\[
\tau_1\steq{def} \inf\{t{>}0: X_2(t){\ge}q\},
\]
before this instant, the process $(X_2(t))$  has the sample paths $(L(t))$  of an $M/M/\infty$ queue, see Section~\ref{MarkovSec},  with arrival rate $\kappa_1 $ and service rate $ \kappa _2$. At time $\tau_1$ the state of the process  has the same distribution as the random variable
\[
(N{+}{\cal N}_{\kappa_1}(0,\tau_1), q),
\]
where ${\cal N}_{\kappa_1}$ is a Poisson process with rate $\kappa_1$ and ${\cal N}_{\kappa_1}(a,b)$ denotes the number of points of this process in the interval $[a,b]$. Clearly $\tau_1$ is integrable as well as the random variable ${\cal N}_{\kappa_1}(0,\tau_1)$. 
Therefore, $\E_{(N,b)}[X_1(\tau_1)]\leq N{+}\kappa_1 C_1 $, for some constant $C_1$.

To summarize, starting from the initial state $x_N{=}(N,b)$ with $b{<}q$,  the quantities $\E_{x_N}(\tau_1)$ and $\E_{x_N}(X_1(\tau_1)){-}N$  are bounded by a constant.  We are thus left to study the following case.

\subsection*{Step~3}
The  initial state is  $x_N{=}(N, q)$ with  $N$ large.

As long as $X_2(t){\geq}q$, the third  reaction is active, $p$ copies of $S_1$ are removed and a copy of $S_2$ is created. Initially its rate is of the order of  $N^{p}$, the fastest reaction rate by far.  We define $\nu$ as the number of jumps  before another reaction takes place. 
\[
\nu \steq{def} \inf\{ n\geq 1: X(t_n){-}X(t_{n-1}){\neq} (-p, 1)\},
\]

\[
\P(\nu{>}k)=\prod_{i=0}^{k-1}\left(1{-}\frac{\kappa_1{+}\kappa_2(q{+}i){+}\kappa_4(q{+}i)^{(q+1)}}{\kappa_3(N{-}pi)^{(p)}(q{+}i)^{(q)}{+}\kappa_1{+}\kappa_2(q{+}i){+}\kappa_4(q{+}i)^{(q+1)}}\right),
\]
with the convention that $q^{(q+1)}{=}0$. For $i{\ge}1$, 
\begin{multline*}
  \frac{\kappa_1{+}\kappa_2(q{+}i){+}\kappa_4(q{+}i)^{(q+1)}}{\kappa_3(N{-}pi)^{(p)}(q{+}i)^{(q)}{+}\kappa_1{+}\kappa_2(q{+}i){+}\kappa_4(q{+}i)^{(q+1)}}
  \\\leq
  \frac{(\kappa_1{+}\kappa_2(q{+}i))(q{+}i)^{-(q+1)}{+}\kappa_4}{\kappa_3(N{-}pi)^{(p)}/i{+}(\kappa_1{+}\kappa_2(q{+}i))(q{+}i)^{-(q+1)}{+}\kappa_4}
\le\frac{iC_0}{(N{-}pi)^{(p)}{+}i C_0 },
\end{multline*}
for some appropriate constant $C_0{>}0$. 
Hence, if we fix $0{<}\delta{<}1/2p$,
\[
\E_{x_N}(\nu)\ge \delta N \P(\nu{>}\delta N)\ge\delta N\left(1{-}\frac{\delta NC_0}{(N{-}p\lfloor \delta N\rfloor)^{(p)}{+}\delta N C_0 }\right)^{\lfloor \delta N\rfloor},
\]
so that, since $p{>}2$,
\begin{equation}\label{Est1}
\liminf_{N\to+\infty} \frac{1}{N}\E_{x_N}(\nu) \ge \delta.
\end{equation}
We define $\tau_2{=}t_{\nu}$, obviously
\[
\E_{x_N}(\tau_2)\le \frac{1}{\kappa_1},
\]
and we have
\[
\E_{x_N}(\|X(\tau_2)\|{-}\|x_N\|)\le (1{-}p)\E_{x_N}(\nu){+}2 \le -\gamma N,
\]
for some $\gamma{>}0$ if $N$ is sufficiently large, using Relation~\eqref{Est1}. Consequently is easy to see that there is a convenient constant $K$ such that Condition~\eqref{LyapCond} holds for this set of initial states and the stopping time $\tau_2$, and also for the initial states of Step~2 and the stopping time $\tau_1{+}\tau_2{\circ}\theta_{\tau_1}$.
The proposition is proved.
\end{proof}

\subsection*{A Scaling Picture}
The key argument of the proof of the positive recurrence is somewhat hidden behind an estimate of the expected value of the hitting time $\nu$ in Step~3. It is not difficult to figure out that, starting from the state $(N,q)$, the ``right'' timescale is $t{\mapsto}t/N^{p+q-1}$.  In this section we sketch a scaling argument to describe in more detail how the norm of the state goes to $0$. It could also give an alternative way to handle Step~3.

Define the Markov jump process $(Z_N(t))= (Z^N_1(t), Z^N_2(t))$ corresponding to the last two reactions of the  CRN network~\eqref{AgMattCRN}. Its $Q$-matrix is given by, for $z{\in}\N^2$,
\begin{equation}\label{QMZ}
z\longrightarrow z{+}\begin{cases}
  {-}pe_1{+}e_2& \kappa_3 z_1^{(p)}z_2^{(q)},\\
  {-}e_2& \kappa_4z_2^{(q+1)},
\end{cases}
\end{equation}
with initial state $(N,q)$. The scaling results of this section are obtained for this process. It is not difficult to show that they also  hold for the CRN network~\eqref{AgMattCRN} since the discarded reactions are on a much slower timescale. 

Define the Markov jump process $(Y_N(t))= (Y^N_1(t), Y^N_2(t))$ whose $Q$-matrix is given by, for $y{\in}\N^2$,
\[
y\longrightarrow y{+}\begin{cases}
  {-}pe_1{+}e_2& \kappa_3 y_1^{(p)},\\
  {-}e_2& \kappa_4(y_2{-}q),
\end{cases}
\]
with the same initial state.  If $p{\ge}2$, with standard arguments, it is not difficult to show the convergence in distribution
\begin{multline}\label{Ageq1}
\lim_{N\to+\infty}\left(\frac{1}{N}\left(Y^N_1,Y^N_2\right)\left(\frac{t}{N^{p-1}}\right)\right)=\left(y_1(t), y_2(t)\right)
\\\steq{def} \left(\frac{1}{\sqrt[p{-}1]{p(p{-}1)\kappa_3 t{+}1}},\frac{1{-}y_1(t)}{p}\right)
\end{multline}
From this convergence we obtain that for any $\eta{\in}(0,1/p)$, if $H^N_Y(\eta)$ is the hitting time  of $[\lfloor \eta N\rfloor,{+}\infty)$ by $(Y^N_2(t))$,   then the sequence $(N^{p-1}H^N_Y(\eta))$ converges in distribution to some constant.

For $t{\ge}0$, define the stopping time
\[
\tau^N_t=\inf\left\{s{>}0: \int_0^s \frac{1}{Y^N_2(u)^{(q)}}\diff u \ge t\right\},
\]
and $(\widetilde{Z}^N(t)){=}(Y^N(\tau^N_t))$, then it is easy to check that $(\widetilde{Z}^N(t))$ is a Markov process whose $Q$-matrix is given by Relation~\eqref{QMZ}. See Section~III.21 of~\citet{Rogers1} for example.  Consequently,  $(\widetilde{Z}^N(t))$ has the same distribution as $(Z^N(t))$.
\begin{proposition}
If $p$, $q{\ge}2$, $(X^N(0)){=}(\lfloor\delta N\rfloor,\lfloor(1{-}\delta)N/p\rfloor)$, for some $\delta{\in}(0,1)$, then for the convergence in distribution
  \[
  \lim_{N\to+\infty} \left(\frac{1}{N}X^N\left(\frac{t}{N^{p+q-1}}\right)\right)=(x_1(t),x_2(t)){=}\left(\left(y_1,\frac{1{-}y_1}{p}\right)\left(\phi^{-1}(t)\right)\right),
  \]
with 
      \[
(y_1(t))=\left(\frac{\delta}{\sqrt[p{-}1]{p(p{-}1)\delta^{p-1}\kappa_3 t{+}1}}\right) \text{ and }  \phi(t)\steq{def} \int_0^t \frac{p^{q}}{(1{-}y_1(s))^q}\diff s.
  \]
\end{proposition}
\begin{proof}
As mentioned above, from this initial state and this timescale, the processes  $(Z^N(t))$ and  $(X^N(t))$ have the same asymptotic behavior for values of the order of $N$. 
The proof uses the convergence~\eqref{Ageq1} and the time-change argument described above.
\end{proof}

The above proposition shows that on a convenient timescale, both coordinates of $(X^N(t))$ are of the order of $N$. The scaled version of the first one is converging to $0$, while the second component is increasing.

If $Y^N(0){=}(\lfloor\delta N\rfloor,\lfloor(1{-}\delta)N/p\rfloor)$, for some $\delta{>}0$, let 
\[
R_N\steq{def}\inf\{t{>}0: Y^N_1(t)\le \sqrt[p]{N}\}. 
\]
By writing the evolution of $(Y^N(t))$ in terms of an SDE like Relation~\eqref{SDESec}, one easily obtains,
\begin{multline*}
\E(Y^N_1(R_N{\wedge}t))=\lfloor\delta N\rfloor{-}p\kappa_3 \ E\left(\int_0^{R_N{\wedge}t}Y^N_1(s)^{(p)}\diff s\right)\\
\le \lfloor\delta N\rfloor{-}p\kappa_3 \left(\lfloor \sqrt[p]{N}\rfloor \right)^{(p)} E\left(R_N{\wedge}t\right),
\end{multline*}
hence
\[
E\left(R_N{\wedge}t \right)\le \frac{\lfloor\delta N\rfloor}{p\kappa_3 (\lfloor \sqrt[p]{N}\rfloor )^{(p)}},
\]
by using the monotone convergence theorem, we obtain that
\[
\sup_{N} E\left(R_N\right)<{+}\infty.
\]
It is easily seen that the same property holds for $(X^N_1(t))$. 

To finish the description of the return path to $(0,0)$, we can assume therefore that $X^N(0){=}(\lfloor\sqrt[p]{N}\rfloor,N)$. 
It is not difficult to see that the reaction $(q{+}1)S_2\mathrel{\mathop{\xrightharpoonup{\kappa_4}}} qS_2$ is driving the evolution as long as $(X^N_2(t))$ is ``large'' since $(X^N_1(t))$ cannot grow significantly on the corresponding timescale. More formally, also with the same arguments as in Section~\ref{ConvT2Prop}, the convergence in distribution
\[
\lim_{N\to+\infty}\left(\frac{1}{N}\left(X^N_1,X^N_2\right)\left(\frac{t}{N^{q}}\right)\right)=\left(0, \frac{1}{\sqrt[q]{1{+}\kappa_4qt}}\right)
\]
holds.  The second coordinate returns to $0$. 
\section{A CRN with Bi-Modal Behavior}\label{CapSec}
In this section, the positive recurrence and scaling properties of the following interesting CRN are investigated
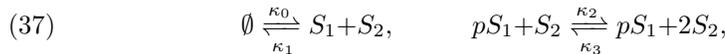
\begin{equation}\label{CapCRN}
\emptyset \mathrel{\mathop{\xrightleftharpoons[\kappa_{1}]{\kappa_{0}}}}  S_1{+}S_2, \hspace{1cm} pS_1{+}S_2 \mathrel{\mathop{\xrightleftharpoons[\kappa_{3}]{\kappa_{2}}}}pS_1{+}2S_2,
\end{equation}
with $p{\ge}2$.

This is an important model  introduced and discussed from the point of view of its stability properties in~\citet{CapConf} for $p{=}2$. The boundary effects are as follows: the second reaction cannot occur if there are less than $p$ copies of $S_1$, and if the number of copies of $S_2$ is zero, only external arrivals change the state of the CRN.  They complicate significantly the verification of a Lyapunov criterion.

We show how Theorem~\ref{ErgCrit} can be used for positive recurrence and that a scaling analysis gives an interesting insight for its time evolution. This is also an example of a CRN with a non-trivial dynamic behavior  which can be investigated with scaling ideas and stochastic calculus involving time change arguments.

Section~\ref{SecTCA}  investigates the positive recurrence properties. It is also an occasion to have an other look at the choice of a Lyapunov function in view of  Condition~\ref{LyapCond} of Theorem~\ref{ErgCrit}. Section~\ref{CapBound} considers the limiting behavior of the sample paths of the CRN with a large initial state close to one of the axes.

The Markov process $(X(t))= (X_1(t), X_2(t))$ associated to this CRN has a $Q$-matrix $Q$ given by, for $x{\in}\N^2$,
\[
x\longrightarrow x{+}\begin{cases}
  e_1{+}e_2& \kappa_0,\\
  -e_1{-}e_2& \kappa_1x_1x_2,
\end{cases}
x\longrightarrow x{+}\begin{cases}
   e_2& \kappa_2 x_1^{(p)}x_2,\\
  {-}e_2& \kappa_3 x_1^{(p)}x_2^{(2)},
\end{cases}
\]
where $e_1$, $e_2$ are the unit vectors of $\N^2$. 

By using the SDE formulation of Section~\ref{SDESec} of the appendix, the associated Markov process can be represented by  the solution $(X(t)){=}(X_1(t),X_2(t))$ of the SDE
\begin{equation}\label{SDECap}
\begin{cases}
  \diff X_1(t)&={\cal P}_{0}((0,\kappa_{0}),\diff t){-}{\cal P}_{1}((0,\kappa_{1}X_1X_2(t{-})),\diff t),\\
  \diff X_2(t)&={\cal P}_{0}((0,\kappa_{0}),\diff t) {-}{\cal P}_{1}((0,\kappa_{1}X_1X_2(t{-})),\diff t)\\
        &\hspace{15mm} {+}{\cal P}_{2}\left(\left(0,\kappa_2 X_1^{(p)}X_2(t{-})\right),\diff t\right)\\
        &\hspace{15mm}  {-}{\cal P}_{3}\left(\left(0,\kappa_{3}X_1^{(p)}X_2^{(2)}(t{-})\right),\diff t\right),
\end{cases}
\end{equation}
where ${\cal P}_{i}$, $i{\in}\{0,1,2,3\}$, are fixed independent Poisson processes on $\R_+^2$ with intensity measure $\diff s{\otimes}\diff t$.

\noindent
{\sc A slow return to $0$.}
The second set of reactions of  this CRN needs $p$ copies of $S_1$ to be active.  If the initial state is $(0,N)$,  copies of $S_1$ are created at rate $\kappa_0$, but they are removed quickly at a rate greater than $\kappa_1N$. The first instant when $p$ copies of $S_1$ are present has an average of the order of $N^{p-1}$. See Lemma~\ref{LemCP}.  At this instant, the number of $S_2$ species is $N{+}p$, and the second coordinate can then decrease, quickly in fact. The network exhibits a kind of bi-modal behavior due to this boundary condition.

Starting from the initial state $x{=}(0,N)$,  the time to decrease $(X_2(t))$ by an amount of the order of $N$ has thus an average of the order of $N^{p-1}$. When $p{>}2$ and if we take the usual norm $\|\cdot\|$ as a Lyapunov function, this results is at odds with the condition~\eqref{LyapCond} of Theorem~\ref{ErgCrit}. This problem could in fact be fixed at the cost of some annoying technicalities. Our approach will be of taking another simple, and somewhat natural, Lyapunov function.  See Section~\ref{SecTCA}. An initial state of the form $(N,0)$ leads also to another interesting boundary behavior.

\subsection{Positive Recurrence}\label{SecTCA}
\begin{proposition}\label{CapPR}
The Markov process $(X(t))$ is positive recurrent.
\end{proposition}

Theorem~\ref{ErgCrit} is used to prove this property. The proof is not difficult but it has to be handled with some care. We will introduce two auxiliary processes with which the process $(X_N(t))$ can be decomposed. One describes the process when the first coordinate is below $p$ and the other when it is larger than $p$. This representation gives a more formal description of the bi-modal behavior mentioned above. Additionally, it will turn out to be helpful to establish the scaling properties of this CRN in Section~\ref{CapBound}. 
For $x{=}(x_1,x_2){\in}\N^2$, we introduce 
\begin{equation}
	f_p(x){=}x_1{+}x_2^p,
\end{equation}
$f_p$ will be our Lyapunov function. The strategy is of analyzing separately the two boundary behaviors. The first one is essentially associated with the initial state $(0,N)$ which we have already seen. The other case is for an initial state of the form $(N,0)$, the problem here is of having the second coordinate positive sufficiently often so that reaction $S_1{+}S_2\rightharpoonup \emptyset$ can decrease significantly the first coordinate.

\subsubsection{Large Initial State along the Horizontal Axis}\label{CAPHA}
In this section it is assumed that the initial state is $x(0){=}(x^0_1,b)$, where $b{\in}\N$ is fixed and $x^0_1$ is ``large''. Without loss of generality one can assume $b{>}0$, otherwise nothing happens until an external arrival.

As long as the second coordinate of $(X(t))$ is  non-null the transitions associated to ${\cal P}_{i}$, $i{=}2$, $3$ occur at a fast rate. When $(X_2(t))$ is $0$, only one chemical reaction may happen, external arrivals and at a ``slow'' rate $\kappa_0$.

We define by induction the non-increasing sequence $(T_k)$ as follows, $T_0{=}0$, and
\[
T_{k+1}=\inf\{t{>}T_k: X_1(t){-}X_1(t{-}){=}{-}1\}. 
\]
The variables $(T_k)$ are stopping times for the underlying filtration $({\cal F}_t)$ defined as in the appendix, see Relation~\eqref{SDEFilt}.

For $t{>}0$, by using the fact that the Poisson process ${\cal P}_{i}$, $i{=}1$, $2$, $3$ are independent and $(X_2(t))$ is greater than $1$ until $T_1$ at least,  we have
\[
\P(T_1{\ge}t)\le \E\left(\exp\left({-}\kappa_1x_1^0\int_0^t X_2(s)\diff s\right)\right)\leq\exp\left({-}\kappa_1x^0_1t\right),
\]
hence $\E(T_1){\le}1/(\kappa_1x^0_1)$. Similarly, with the strong Markov property, for $1{\le}k{<}x^0_1$,
\[
\E(T_{k+1}{-}T_k){\le}\frac{1}{\kappa_0}{+}\frac{1}{\kappa_1(x^0_1{-}k)},
\]
the additional term $1/\kappa_0$ comes from the fact that $X_2(T_k)$ can be zero, so that one has to wait for an exponentially distributed amount of time with parameter $\kappa_0$ to restart the CRN. 

For $n_0{\ge}1$, we have seen that the random variable $T_{n_0}$ is stochastically bounded by the sum of $2n_0$ i.i.d.  exponentially distributed random variables with the same positive parameter, hence
\[
C_0\steq{def} \sup_{x_1^0>n_0}\E_{x^0_1}(T_{n_0}) < {+}\infty
\]

Let ${\cal E}_1$ be the event when ${\cal P}_{1}$ has a jump before ${\cal P}_{0}$ in SDE~\eqref{SDECap}, then
\[
\P({\cal E}_1^c){\le}\frac{\kappa_0}{\kappa_1x^0_1{+}\kappa_0}.
\]
Similarly, for $k{\ge}2$,  ${\cal E}_k$ is a subset of the event ${\cal E}_{k-1}$ for which ${\cal P}_{1}$ has a jump before ${\cal P}_{0}$ after the first time after $T_k$ when $(X_2(t))$ is greater than $1$, then
\begin{equation}\label{CapEqE}
\P_{x_1^0}({\cal E}_k^c){\le}\sum_{i=0}^{k-1}\frac{\kappa_0}{\kappa_1(x^0_1{-}i){+}\kappa_0}{\le}\frac{\kappa_0 k}{\kappa_1(x^0_1{-}k){+}\kappa_0}.
\end{equation}

Let $s_1$ be the first instant of jump of ${\cal P}_{0}((0,\kappa_0){\times}(0,t])$. From $t{=}0$, as long as the point process  ${\cal P}_{0}$, does not jump in SDE~\eqref{SDECap}, that is, on the time interval $[0,s_1]$,  up to a change of time scale $t{\mapsto}X_1X_2(t)$, the process $(X_1(t),X_2(t))$ has the same sequence of visited states as the solution $(Y(t))$ of the SDE 
\begin{equation}\label{Yeq2}
\begin{cases}
  \diff Y_1(t)&= {-}{\cal P}^Y_{1}((0,\kappa_{1}),\diff t),\\
  \diff Y_2(t)&= {-}{\cal P}^Y_{1}((0,\kappa_{1})),\diff t)\\
        &\hspace{15mm} {+}{\cal P}^Y_{2}\left(\left(0,\kappa_2Y_1(t{-})^{(p)-1}\right),\diff t\right)\\
        &\hspace{15mm}  {-}{\cal P}^Y_{3}\left(\left(0,\kappa_{3}Y_1^{(p)-1}(Y_2(t{-}){-}1)^+\right),\diff t\right),
\end{cases}
\end{equation}
with the same initial state and the slight abuse of notation $y^{(p){-}1}{=}y^{(p)}/y$. The random variables ${\cal P}^Y_{i}$, $i{\in}\{1,2,3\}$, are independent Poisson processes on $\R_+$ independent of ${\cal P}_{i}$, $i{\in}\{0,1,2,3\}$ with the same distribution as ${\cal P}_0$. 
 In particular if $u_1$ is the first instant when $(Y_1(t))$ has a downward jump, an independent exponential random variable with parameter $\kappa_1$,  then the relation $Y_2(u_1){=}X_2(T_1)$ holds on the event $\{T_{1}{\le}s_1\}$.

 From $t{=}0$, as long as the first coordinate of $(Y_1(t))$ does not change, the second component $(Y_2(t))$ has the same distribution as $(L_b((x_1^0)^{(p)-1}t))$, where $(L_b(t))$ is a birth and death process with birth rate $\kappa_2 $ and death rate $\kappa_3(x-1)$, for $x{\ge}1$ and  initial state $b$. The process $(L_b(t))$ can be expressed with the process of an $M/M/\infty$ queue. It is easily seen $(\E(L_b(t)^p))$ is a bounded function, consequently,
\begin{equation}\label{CapC0}
\sup_{x(0)} E\left(X_2(T_{1})^p\right) \le C_1 {<} {+}\infty,
\end{equation}
by induction, the same result holds for $T_{n_0}$ for a convenient constant $C_1$.

Note that if $X_2(T_1{-}){=}1$, the next chemical reaction after time $T_1$ will be $$\emptyset \rightharpoonup S_1+S_2,$$
and therefore the downward jump of $X_1$ will be canceled.

At time $T_1$, a downward jump of  the process $(X_1^N(t))$ is possible if it happens when $X_2(T_1{-})\geq 2$ i.e. if $L_b((x_1)^{(p)-1}u_1){\ne}0$.
It is easy to construct a coupling of the processes $(L_b(t))$ and $(L_0(t))$, such that the relation $L_b(t){\ge}L_0(t)$ holds for all $t{\ge}0$. The convergence of $(L_0(t))$ to equilibrium gives the existence of $K_0{\ge}0$ and $\eta_0{>}0$ such that if $x^0_1{\ge}K_0$ then $\P(L_0((x_1)^{(p)-1}u_1){>}0){\ge}\eta_0$. 

We can gather these results, by using the stochastic bound on $T_{n_0}$, we obtain the relations
\begin{align*} E_{x(0)}&(f_p(X(T_{n_0}))){-}f_p(x(0))\le {-}{n_0}\eta_0\P_{x(0)}({\cal E}_{n_0})\\
  &  +\E_{x(0)}\left({\cal P}_{0}\left((0,\kappa_0){\times}(0,T_{n_0}]\ind{{\cal E}_{n_0}^c}\right)\right)     +E\left(X_2(T_{n_0})^p\right){-}b^p\\
    &\le {-}\eta_0n_0+n_0\eta_0\P_{x(0)}({\cal E}_{n_0}^c){+}\kappa_0C_0 {+}C_1.
\end{align*}
One first choose $n_0$ so that $n_0{>}3(\kappa_0C_0 {+}C_1)/\eta_0$ and then with Relation~\eqref{CapEqE},  $K_1{\ge}K_0$ such that
$n_0\eta_0\P_{K_1}({\cal E}_k^c){<}(\kappa_0C_0 {+}C_1)$. We obtain therefore that if $x^0_1{>}K_1$, then
\begin{equation}\label{CapHAPR}
\E_{x(0)}\left(f_p\left(X(T_{n_0})\right){-}f_p(x(0))\right)\le -\delta,
\end{equation}
for some $\delta{>}0$ and $\sup(\E_{x(0)}(T_{n_0}): x_1{\ge}K){<}{+}\infty$. 
Relation~\eqref{CapHAPR} shows that Condition~\eqref{LyapCond} of Theorem~\eqref{ErgCrit} is satisfied for our Lyapunov function $f_p$ and stopping time $T_{n_0}$ for the initial state of the form $(x^0_1,b)$.

\subsubsection{Initial State with a Large Second Component}\label{CAPVA}
In this section it is assumed that the initial state is $x(0){=}(a,x^0_2)$ with  $a{<}p$ and $x^0_2$ is large. We note that, as long as $(X_1(t))$ is strictly below $p$, the two coordinates experience the same jumps, the quantity $(X_2(t){-}X_1(t))$ does not change.

For $x{\ge}0$ and $k{\le}p-1$, we introduce the process  $(Z(k, x, t))$ the solution of the SDE
\begin{equation}\label{SDECapZ}
  \diff Z(k, x^0_2,t)={\cal N}_{\kappa_0}(\diff t) {-}{\cal P}_{Z}((0,\kappa_{1}Z(k, x^0_2,t{-})(x^0_2{-}k{+}Z(k, x^0_2,t{-})))),\diff t),\\
\end{equation}
with $Z(k, x^0_2,0){=}k$ and ${\cal P}_Z$ is a Poisson process on $\R_+^2$ and ${\cal N}_{\kappa_0}$ is a Poisson process on $\R_+$ with parameter $\kappa_0$. This process will be used to represent $(X(t))$ when its first coordinate is less than $p{-}1$. I

For $z{<}p$, we define
\[
	S_Z(z, x^0_2)\steq{def}\inf\{t{>}0: Z(z, x^0_2,t){=}p\},
\]
if $X(0){=}(0,x^0_2)$, then  it is easily seen that the relation
\[
(X(t{\wedge}S_Z(0, x^0_2))){\steq{dist}}(Z(0,x^0_2, t{\wedge}S_Z(0,x^0_2)),x^0_2{+}Z(0, x^0_2, t{\wedge}S_Z(0, x^0_2)))
\]
holds by checking the jump rates.

We define, for $x{=}(x_1,x_2){\in}\N^2$, 
\[
\lambda(x)=\kappa_0{+}\kappa_1x_1x_2{+}\kappa_2x_1^{(p)}x_2{+}\kappa_3x_1^{(p)}x_2^{(2)},
\]
it is the total  jump rate of $(X(t))$ in state $x$.

\begin{lemma}\label{LemCP}
  For $x_1^0{\ge}\kappa_0/(\kappa_1p)$,
  \[
		\limsup_{x^0_2\rightarrow +\infty}\frac{\E(S_{Z}(0,x^0_2))}{(x_2^0)^{p-1}}\le C_2,
  \]
  for some constant $C_2$.
\end{lemma}
\begin{proof}
A simple coupling shows that the  process $(Z(0,x, t)$  stopped at time $S_Z(0,x)$ is lower bounded by a birth and death process $(U(t))$ starting at $0$ with, in state $x$, a birth rate $\kappa_0$ and a death rate $a_1{=}\kappa_1p(x+p)$. Denote by $H$ the hitting time of $p$ by $(U(t))$, then it is easily seen, that, for $0{<}k{<}p$,
  \[
  (\E_k(H){-}\E_{k+1}(H))=\frac{a_1}{\kappa_0}(\E_{k-1}(H){-}\E_k(H))+\frac{1}{\kappa_0},
  \]
with $\E_{0}(H){-}\E_1(H){=}1/\kappa_0$.In particular $\E(H_{Z}(0,x)){\le}\E_{0}(H)$. We derive the desired inequality directly from this relation.
\end{proof}

\begin{enumerate}
\item If $x_1{\ge}p$.\\
  Define
\[
	C_1\steq{def}\sup_{x_2{\ge}1} \left(\frac{(x_2{+}p)^{(p)}{-}(x_2)^{(p)}}{x_2^{p-1}}\right)<{+}\infty
\]
and
\[
\tau_1\steq{def}\inf\{t{>}0: \Delta X_1(t){+}\Delta X_2(t)\neq{-}1\},
\]
where $\Delta X_i(t){=}X_i(t){-}X_i(t{-})$, for $i{\in}\{1,2\}$ and $t{\ge}0$. The variable $\tau_1$ is the first instant when a reaction other than $pS_1{+}2S_2 \mathrel{\mathop{\rightharpoonup}}pS_1{+}S_2$ occurs.

For $1{\le}k_0{<}x_2$, then
\[
\P_{x(0)}(X_2(\tau_1)\le x_2{-}k_0-1)\geq \prod_{i=0}^{k_0}\frac{\kappa_3 x_1^{(p)}(x_2{-}i)^{(2)}}{\lambda((x_1,x_2{-}i))}
\geq p_{k_0}\steq{def}\prod_{i=0}^{{k_0}}\frac{\kappa_3 p^{(p)}(x_2{-}i)^{(2)}}{\lambda((p,x_2{-}i))}
\]
and there exists $K_{0}{\ge}k_0$ such that if $x_2{\ge}K_{0}$, then
\[
(x_2{-}k_0)^{(p)}{-}(x_2)^{(p)}\le {-}\frac{k_0}{2}x_2^{p-1}\text{ and } p_{k_0}\ge \frac{1}{2},
\]
from these relations, we obtain the inequality
\begin{multline}\label{eqLa}
  \E_{x(0)}\left(f_p\left(X(\tau_1)\right){-}f_p(x)\right)\leq 1{+}\left((x_2{-}{k_0})^{(p)}{-}x_2^{(p)}\right)p_{k_0} {+}\left((x_2{+}1)^{(p)}{-}x_2^{(p)}\right)\\
  \le \left(-\frac{k_0}{4}{+}1{+}C_1\right)x_2^{p-1}.
\end{multline}
We choose $k_0{=}\lceil 4(3{+}2C_1)\rceil$, hence, for $x_2{\ge}K_{0}$ the relation
\[
\E_{x(0)}\left(f_p\left(X(\tau_1)\right){-}f_p(x)\right)\leq {-}2x_2^{p-1}. 
\]
holds, and note that $\E(\tau_1){\le}1/\kappa_0$.

\item  If $x_1{\le}p{-}1$.\\
  Define
	\[
		\nu_0=\inf\{t{>}0: X_1(t){\ge}p\},
	\]
	When $x_1{=}0$, the variable $\nu_0$ has the same distribution as $S_Z(0, x_2)$. Otherwise, if $x_1{>}0$, it is easily seen that $\E_{x(0)}(\nu_0){\leq} \E(S_Z(0,x_2))$. 
	Lemma~\ref{LemCP} gives the existence of  a constant $C_2{>}0$ so that  
	\[
		\sup_{x_2\geq K_{0}}\left( \frac{\E_x(\nu_0)}{x_2^{p-1}}\right) < C_2. 
	\]
	The state of the process at time $\nu_0$ is $X(\nu_0){=}(p,x_2+(p-x_1))$, in particular
	\[
		\E_{x(0)}\left(f_p(X(\nu_0)){-}f_p(x)\right)\le p{+}C_1 x_2^{p-1},
	\]
	and at that time, we are in case a).

	The convenient stopping time  is defined  as $\tau_2{\steq{def}}\nu_0{+}\tau_1(\theta_{\nu_0})$, where $\theta_a$ is the operator of the classical time shift by $a$.  With $k_0$ and $K_{0}$ as before,  if $x_2{\ge}K_{0}$, by using Relation~\eqref{eqLa}, we obtain the inequality
\begin{multline*}
		\E_{x}\left(f(X(\tau_2))-f(x)\right)\leq p{+}C_1x_2^{p-1} {+}\E_{(p, x_2+(p-x_1))}\left(f(X(\tau_1))-f(x)\right) \\\leq p{+}C_1x_2^{p-1} {+}\left(1{+}C_1-\frac{k_0}{4}\right)\left(x_2+(p-x_1)\right)^{p-1}\le {-}x_2^{p-1}
\end{multline*}
holds.
\end{enumerate}
\begin{proof}[Proof of Proposition~\ref{CapPR}]
Theorem~\ref{ErgCrit} can be used as a consequence of a), b), and Relation~\eqref{CapHAPR}.
\end{proof}

\subsection{A Scaling Picture}\label{CapBound}
We investigate the scaling properties of $(X_N(t))$ when the initial state is of the form $(N,0)$ or $(0,N)$ essentially. In the first case, an averaging principle is proved on a convenient timescale. A time change argument is an important ingredient to derive the main convergence result.

In the second case, the time evolution of the second coordinate of the process is non-trivial only on ``small'' time intervals but with a ``large'' number of jumps, of the order of $N$. This accumulation of jumps has the consequence that the convergence of the scaled process cannot hold with the classical Skorohod topology on ${\cal D}(\R_+)$. There are better topologies to handle properly this kind of situation. To keep the presentation simple, we have chosen to work with a weaker topology, the weak convergence on the space of random measures, to study the weak convergence of the  occupation measures of the sequence of scaled processes. See~\citet{Dawson}. 

\subsubsection{Horizontal Axis}
For $N{\ge}1$, the initial state is $(x^N_1,b)$, $b{\in}\N$ is fixed, it is assumed that
\begin{equation}\label{CapInitH}
\lim_{N\to+\infty}\frac{x_1^N}{N}=\alpha_1{>}0.
\end{equation}
When the process $(X_2(t))$ hits $0$, it happens only for a jump of ${\cal P}_{1}$. In this case only a jump of ${\cal N}_{\kappa_0}$ restarts the activity of the CRN. 

We introduce the process $(Y_N(t)){=}(Y_1^N(t),Y_2^N(t))$, solution of the SDE,
\begin{equation}\label{HaY}
\begin{cases}
  \diff Y^N_1(t)&={\cal N}_{\kappa_0}(\diff t) {-}\ind{Y^N_2(t{-}){>}1}{\cal P}^Y_{1}((0,\kappa_{1}Y^N_1Y^N_2(t{-})),\diff t),\\
  \diff Y^N_2(t)&={\cal N}_{\kappa_0}(\diff t) {-}\ind{Y^N_2(t{-}){>}1}{\cal P}^Y_{1}((0,\kappa_{1}Y^N_1Y^N_2(t{-})),\diff t)\\
        &\hspace{15mm} {+}{\cal P}^Y_{2}\left(\left(0,\kappa_2(Y^N_1)^{(p)}Y^N_2(t{-})\right),\diff t\right)\\
        &\hspace{15mm}  {-}{\cal P}^Y_{3}\left(\left(0,\kappa_{3}(Y^N_1)^{(p)}Y^N_2(t{-})^{(2)}\right),\diff t\right),
\end{cases}
\end{equation}
with initial condition $(Y_1^N(0),Y_2^N(0)){=}(x_1^N,b)$.

The process $(Y_N(t))$ behaves as $(X(t))$ except that its second coordinate cannot be $0$ because the associated transition is excluded. 
In state $(x,1)$ for $(X(t))$,  if a jump of the Poisson process ${\cal P}_{1}^Y$ occurs, the state becomes $(x{-}1,0)$, see   Relation~\eqref{SDECap}. It stays in this state for a duration which is exponentially distributed with parameter $\kappa_0$. After that time the state of $(X(t))$ is back to $(x,1)$. These time intervals during which $(X_2(t))$ is $0$ are, in some sense, removed to give $(Y_N(t))$. This is expressed rigorously via a time change argument. See Chapter~6 of~\citet{KurtzEthier} for example.
 
Now the strategy to obtain a scaling result for $(X^N_1(t)$ is of establishing  a convergence result for $(Y_N(t))$ and, with an appropriate change of timescale, express the process $(X^N_1(t))$ as a ``nice'' functional of $(Y_N(t))$.  This is the main motivation of the introduction of $(Y_N(t))$ which describes of the two regimes of this CRN.

Define
\[
\left(\overline{Y}_1^N(t)\right)=\left(\frac{Y_1^N(t)}{N}\right)
\text{ and }
\croc{\mu_N,f}\steq{def}\int_0^{+\infty}f(s,Y_2^N(s))\diff s,
\]
if $f$ is a function on $\R_+{\times}\N$ with compact support, $\mu_N$ is the {\em occupation measure} associated to $(Y^N_2(t))$. 

\begin{proposition}\label{SACp}
The sequence  $(\mu_N,(\overline{Y}_1^N(t)))$ is converging in distribution to a limit $(\mu_\infty,(y_\infty(t))$ defined by
\[
\croc{\mu_\infty,f}=\int_{\R_+\times\N} f(s,x)\pi_Y(\diff x)\diff s,
\]
if $f{\in}{\cal C}_c(\R_+{\times}\N)$,  the function $(y_\infty(t))$ is given by
\begin{equation}\label{ODEC}
y_\infty(t) = \alpha_1\exp\left( -\frac{\kappa_1\kappa_2}{\kappa_3} t\right) \text{  for  }t{\geq}0,
\end{equation}
 and $\pi_Y$ is the distribution on  $\N{\setminus}\{0\}$ defined by, for $x{\ge}1$, 
\[
\pi_Y(x) = \frac{1}{x!}\left(\frac{\kappa_2}{\kappa_3}\right)^x\frac{1}{e^{\kappa_2/\kappa_3}{-}1}. 
\]
\end{proposition}
\begin{proof}
The proof is quite standard. See~\citet{Kurtz1992}.  Because of the term $Y^N_1(t)^{(p)}$ in the SDE of the process $(Y^N_2(t))$, the difficulty is to take care of the fact that $(Y^N_1(t))$ has to be of the order of $N$, otherwise $(Y^N_2(t))$ may not be a ``fast'' process. We give a sketch of this part of the proof.
  
Let $a$, $b{\in}\R_+$ such that $0{<}a{<}\alpha_1{<}b$, and 
\[
S_N=\inf\left\{t{>}0,\overline{X}^N_1(t){\not\in}(a,b)\right\}.
\]
Let $(L(t))$ a birth and death process on $\N$, when in state $y{\ge}1$, its birth rate is $\beta y$ and the death rate is $\delta y(y{-}1)$, with $\beta{=}(\kappa_0{+}\kappa_2b^p)$ and $\delta{=}\kappa_3a^p$. Its invariant distribution is a Poisson distribution with parameter $\beta/\delta$ conditioned to be greater or equal to $1$.

If $N$ is sufficiently large, we can construct a coupling of  $(Y^N_2(t))$ and $(L(t))$, with $L(0){=}Y^N_2(0)$ and such that the relation
\[
Y^N_2(t){\le}L(N^{p}t)
\]
holds for $t{\in}[0,S_N)$.

For $t{>}0$,
\[
\frac{Y^N_1(t)}{N}\ge \frac{x_1^N}{N}{-}\kappa_1\int_0^t Y_1^N(s)Y_2^N(s)\,\diff s {-}M_Y^N(t),
\]
where $(M_Y^N(t))$ is the martingale given by 
\[
\left(\frac{1}{N}\int_0^t \ind{Y^N_2(s{-}){>}1}\left[{\cal P}^Y_{1}((0,\kappa_{1}Y^N_1(s{-})Y^N_2(s{-})),\diff s)
{-}\kappa_{1}Y^N_1(s)Y^N_2(s))\diff s\right]\right),
\]
we have 
\begin{equation}\label{SAeq2}
\frac{Y^N_1(t{\wedge}S_N)}{N}\ge \frac{x_1^N}{N}{-}\kappa_1b\int_0^t L(N^ps),\diff s {+}M_Y^N(t{\wedge}S_N),
\end{equation}
and 
\[
\croc{M_Y^N}(t{\wedge}S_N)\le \frac{b}{N}\int_0^tL(N^ps)\diff s.
\]
By the ergodic theorem applied to $(L(t))$, almost surely
\begin{multline*}
  \lim_{N\to+\infty} \int_0^t L(N^ps)\diff s=\lim_{N\to+\infty} \int_0^t \E(L(N^ps))\diff s\\
  =\lim_{N\to+\infty} \frac{1}{N^p}\int_0^{N^pt} L(s),\diff s=\frac{\beta}{\delta}\frac{t}{1{-}\exp({-}\beta/\delta)}.
\end{multline*}
We deduce that  $(M_Y^N(t),t\le \eta)$ is converging in distribution to $0$ by Doob's Inequality and, with Relation~\eqref{SAeq2}, that there exists $\eta{>}0$ such that
\begin{equation}\label{SAeq1}
\lim_{N\to+\infty}\P(S_N{>}\eta)=1.
\end{equation}
For $\eps{>}0$ and $K{>}0$, 
\begin{multline*}
\E(\mu_N([0,\eta]{\times}[K,{+}\infty]))\le
\E\left(\int_0^{\eta{\wedge}S_N} \ind{Y^N_2(s)\ge K}\diff s\right){+}\eta\P(S_N\leq \eta)
\\  \le \E\left(\int_0^{\eta}\ind{L(N^ps)\ge K}\right){+}\eta\P(S_N\leq \eta),
\end{multline*}
  again with the ergodic theorem and Relation~\eqref{SAeq1}, there exists some  $N_0$ and $K{>}0$ such that
$\E(\mu_N([0,\eta]{\times}[K,{+}\infty])){\le}\eps$. Lemma~1.3 of~\citet{Kurtz1992} shows that the sequence of random measures $(\mu_N)$ on $\R_+{\times}\N$ restricted to $[0,\eta]{\times}\N$ is tight. 

  From there, it is not difficult to conclude the proof of the proposition, on $[0,\eta]$ and extend by induction this result on  the time interval $[0,k\eta]$, for any $k{\ge}1$.
\end{proof}
Let $\widetilde{\cal N}$ be a Poisson process on $\R_+^3$, independent of the Poisson processes $(\cal{P}_{i}^Y)$, whose intensity measure is $\diff s{\otimes}\diff t{\otimes}\kappa_0\exp(-\kappa_0a)\diff a$.  Recall that such a point process has the same distribution as
\[
\sum_{n{\ge}0}\delta_{(s_n,t_n,E_n)},
\]
where $(s_n)$ and $(t_n)$ are independent Poisson processes on $\R_+$ with rate $1$, independent of the i.i.d. sequence $(E_n)$ of exponential random variables with parameter $\kappa_0$. See Chapter~1 of~\cite{Robert}. 

\begin{definition}[Time Change]\label{ANdef}
Define the process $(A_N(t))$ by
\[
	A_N(t) \steq{def}\left(t+\int_{[0,t]\times\R_+} a\ind{Y^N_2(s{-})=1}\widetilde{\cal N}((0,\kappa_{1}Y^N_1(s{-})Y^N_2(s{-})),\diff s,\diff a)\right),
\]
and its associated inverse function as 
\[
	B_N(t)\steq{def}\inf\left\{s>0: A_N(s)\geq t\right\}.
\]
\end{definition}
The instants of jump of $(A_N(t))$ are the instants when $(Y^N_2(t))$ can switch from $1$ to $0$ for the dynamic of $(X^N_2(t))$ and the size of the jump is the duration of time when $(X^N_2(t))$ stays at $0$, its distribution is exponential with parameter $\kappa_0$. 

The process $(A_N(t))$ gives  in fact the correct timescale to construct the process $(X_N(t))$ with the process $(Y_N(t))$.
We define the process $(\widetilde{X}_N(t))$ on $\N^2$ by, for $t{\ge}0$,
\begin{equation}\label{Xdef}
\begin{cases}
  \widetilde{X}_N(A_N(t))=Y_N(t),\\
\displaystyle   \left(\widetilde{X}^N_1(u), \widetilde{X}^N_2(u)\right)=\left(Y^N_1(t{-}){-}1,0\right),\vspace{3mm} u{\in}[A_N(t-),A_N(t)).
\end{cases}
\end{equation}
If $t$ is  instant of jump of $(A_N(t))$, the process does not change on the time interval $[A_N(t-),A_N(t))$. In this way, $(\widetilde{X}_N(t))$ is defined on $\R_+$. 
\begin{lemma}\label{LemX}
For $t{>}0$, then $A_N(B_N(t)){=}t$ if $t$ is not in an interval $[A_N(u-),A_N(u))$ for some $u{>}0$, 
and  the relation
  \[
  \sup_{t\ge 0} |  \widetilde{X}_N(t){-}Y_N(B_N(t))|\le 1
  \]
  holds.
\end{lemma}
\begin{proof}
This is easily seen by an induction on the time intervals $[A_N(s_n),A_N(s_{n+1}))$, $n{\ge}0$, where $(s_n)$ is the sequence of the instants of jump of $(A_N(t))$, with the convention that $s_0{=}0$. 
\end{proof}
\begin{proposition}
The processes  $(X_N(t))$ and $(\widetilde{X}_N(t))$ have the same distribution.
\end{proposition}
\begin{proof}
The proof is standard.  See Chapter~6 of~\citet{KurtzEthier} for example. The Markov property of $(\widetilde{X}_N(t))$ is a consequence of the Markov property of $(Y_N(t))$ and the strong Markov property of the Poisson process $\widetilde{\cal N}$.  It is easily checked that the $Q$-matrices of $(X_N(t))$ and $(\widetilde{X}_N(t))$ are the same. 
\end{proof}
\begin{proposition}\label{convA}
	For the convergence in distribution, 
	\begin{equation}
		\lim_{N\rightarrow +\infty}\left(\frac{A_N(t)}{N}\right) = (a(t))\steq{def} \left(\alpha_1\frac{1}{\kappa_0(e^{\kappa_2/\kappa_3}-1)}\left(1{-}\exp\left({-}\frac{\kappa_1\kappa_2}{\kappa_3}t\right) \right)\right).
	\end{equation}
\end{proposition}
\begin{proof}
  Let $T>0$. By using the fact that, for $0{\le}u{\le}T$,  the relation
  \[
  Y^N_1(u){\le}x_1^N{+}{\cal P}^Y_{1}((0,\kappa_0){\times}(0,T])
\]
holds,   the sequence of processes
    \[
\left(\frac{1}{N}\int_0^t \frac{\kappa_1}{\kappa_0}\ind{Y^N_2(u){=}1}Y^N_1(u)\diff u\right)
    \]
  is thus  tight by the criterion of modulus of continuity.  See Theorem~7.3 of~\citet{Billingsley} for example. Proposition~\ref{SACp} shows that its limiting point is necessarily  $(a(t))$.
    
We note that the process 
    \[
    (M_{A,N}(t))=\left(\frac{1}{N}\left(A_N(t)-t-\frac{\kappa_1}{\kappa_0}\int_0^t \ind{Y^N_2(u){=}1}Y^N_1(u)\diff u\right)\right),
      \]
it is a square integrable martingale whose predictable increasing process is
      \[
      \left(\croc{M_{A,N}}(t)\right)=\left(\frac{\kappa_1}{\kappa_0 N}\int_0^t \ind{Y^N_2(u){=}1}\frac{Y^N_1(u)}{N}\diff u\right).
      \]
The martingale is vanishing as $N$ gets large by Doob's Inequality. The proposition is proved. 
\end{proof}

Proposition~\ref{SACp} establishes a convergence  result for the sequence of processes $(Y^N_1(t)/N)$. In our construction of $(X^N_1(t))$,  time intervals, whose durations are exponentially distributed, are inserted. During these time intervals, the coordinates of the process do not change. To have a non-trivial convergence  result  for $(X^N_1(t)/N)$, the timescale of the process has to be sped-up. It  turns out that the convenient timescale for this is $(Nt)$, this is a consequence of the convergence in distribution of $(A_N(t)/N)$ established in Proposition~\ref{convA}.

\begin{proposition}\label{InvAN}
For the convergence in distribution, the relation
	\begin{multline}\label{eqBlim}
		\lim_{N\rightarrow +\infty}\left(B^N(Nt), t{<}t_\infty\right) = (a^{-1}(t), t{<}t_\infty)\\=\left({-}\frac{\kappa_3}{\kappa_1\kappa_2}\ln\left(\frac{\alpha_1{-}\kappa_0(e^{\kappa_2/\kappa_3}{-}1)t}{\alpha_1}\right), t{<}t_\infty\right),
        \end{multline}
        holds, where $(a(t))$ is defined in Proposition~\ref{convA} and
        \[
        t_\infty=\frac{\alpha_1}{\kappa_0(\exp(\kappa_2/\kappa_3){-}1)}. 
        \]
\end{proposition}
\begin{proof}
  Note that both $(A_N(t))$ and $(B_N(t))$ are non-decreasing processes and that the relation $A_N(B_N(t)){\ge}t$ holds for all $t{\ge}0$.

  We are establishing the tightness property with the criterion of the modulus of continuity. 
  The constants $\eps{>}0$, $\eta{>}0$ are fixed.  For  $0{<}T{<}t_\infty$ we can choose $K{>}0$ sufficiently large so that $a(K){>}T$ and  we define
  \[
  h_K{=}\inf_{s{\le}K}\left(a(s{+}\eta){-}a(s)\right),
  \]
 clearly $h_K{>}0$. By definition of $(B_N(t))$,  we have 
\[
  \P(B_N(NT){\ge} K)=\P\left(\frac{A_N(K)}{N}\le T\right).
  \]
The convergence of Proposition~\ref{convA} shows that  there exists $N_0$ such that if $N{\ge}N_0$, the right-hand side of the last relation is less that $\eps$ and that 
\begin{equation}\label{eqhK}
  \P\left(\sup_{0{\le}u{\le}K}\left|\frac{A_N(u{+}\eta){-}A_N(u)}{N}{-}(a(u{+}\eta){-}a(u))\right|\ge \frac{h_K}{2} \right)\le \eps
\end{equation}
  holds.
  
  For  $\eta{>}0$, and $0{\le}s{\le}t{\le}T$, if $B_N(Nt){-}B_N(Ns){\ge}\eta$ holds, then
  \[
  A_N\left(B_N(Ns){+}\eta\right){-}  A_N\left(B_N(Ns)\right) \le N(t{-}s)
  \]
and if $\delta{\le}h_K/4$, for $N{\ge}N_0$,
\begin{multline*}
  \P\left(\sup_{\substack{0{\le}s\le t{\le}T\\t-s\le\delta_0}}\left|B_N(Nt){-}B_N(Ns)\right|{\ge}\eta \right)\\\le
  \eps{+}\P\left(\inf_{0{\le}u{\le}K} \left(\frac{A_N\left(u{+}\eta\right)}{N}{-} \frac{A_N\left(u\right)}{N}\right) \le \frac{h_K}{4}\right)\le 2\eps,
\end{multline*}
by Relation~\eqref{eqhK}. The sequence of processes $(B_N(Nt))$ is therefore tight and any of its limiting points is a continuous process. The convergence of Proposition~\ref{convA} shows that a limiting point has the same finite marginals as the right-hand side of Relation~\eqref{eqBlim}. The proposition is proved. 
\end{proof}

\begin{theorem}\label{TheoCapH}
  If $(X_N(t)){=}(X^N_1(t),X^N_2(t))$ is the Markov process associated to the CRN~\eqref{CapCRN} whose initial state is  $(x^N_1,b){\in}\N^2$, $b{\in}\N$ and
\[
\lim_{N\to+\infty}{x_1^N}/{N}=\alpha_1{>}0,
\]
then, the  convergence in distribution 
	\[
	\lim_{N\to+\infty}\left(\frac{X^N_1(Nt)}{N}, t{<}t_\infty\right)
	=\left(1{-}\frac{t}{t_\infty},  t{<}t_\infty \right).
	\]
holds, with $t_\infty{=}\alpha_1/(\kappa_0(\exp(\kappa_2/\kappa_3){-}1))$.
\end{theorem}
\begin{proof}
Propositions~\ref{SACp} and~\ref{InvAN} show that the sequence of processes
  \[
  \left(\left(\frac{Y_N(t)}{N},t{>}0\right),\left(B_N(Nt),t{<}t_\infty\right)\right)
  \]
  is converging in distribution to $((y_\infty(t)),(a^{-1}(t),t{<}t_\infty))$. Consequently, the relation
  \[
  \lim_{N\to+\infty}  \left(\frac{Y_N(B_N(Nt))}{N},t{<}t_\infty\right)=\left(y_{\infty}\left(a^{-1}(t)\right),t{<}t_\infty\right)
  \]
  holds for the convergence in distribution. We conclude the proof of the proposition by using Lemma~\ref{LemX}.
\end{proof}

\subsubsection{Vertical Axis}
For $N{\ge}1$, the initial state is $x_N(0){=}(a, x^N_2)$, it is assumed that $a{<}p$ and 
\begin{equation}\label{CapInitV}
\lim_{N\to+\infty}\frac{x_2^N}{N}=1.
\end{equation}
As seen in Section~\ref{CAPVA} when the first coordinate is strictly less than $p$, with a second coordinate of the order of $N$,  it takes an amount of time of the order of $N^{p-1}$  for the process $(X_1^N(t))$ to hit $p$. See Lemma~\ref{LemCP}. In a second, short phase,  a decrease of the second coordinate takes place before returning below $p$. We now establish two convergence results.

\begin{lemma}\label{TZlim}
  If $(Z(z,N,t))$ is the solution of the SDE~\eqref{SDECapZ} with initial state $z{<}p$, and $S_Z(z,N)$ is its hitting time of $p$ then, the sequence $(S_Z(z,N)/N^{p-1})$  converges in distribution to an exponential random variable with parameter
\begin{equation}\label{eqr1}
r_1\steq{def} \frac{\kappa_0}{(p{-}1)!}  \left(\frac{\kappa_0}{\kappa_1}\right)^{p-1}\hspace{-5mm}.
\end{equation}
\end{lemma}
\begin{proof}
 The proof is standard. It can be done by induction on $p{\ge}2$ with the help of the strong Markov property of $(Z(z,N,t))$ for example. 
\end{proof}
We now study the phase during which $(X^N_1(t))$ is greater or equal to $p$. 
Define $(T^N_k)$ the non-decreasing sequence of stopping time as follows, $T^N_0{=}0$ and, for $k{\ge}0$,
\begin{equation}\label{eqTk}
T^N_{k+1}=\inf\{t{\ge} T^N_k: X^N_1(t){=}p{-}1,X^N_1(t{-})=p\}.
\end{equation}
\begin{proposition}[Decay of $(X^N_2(t))$]\label{Decay}
Under Assumption~\eqref{CapInitV} for the initial condition,  for the convergence in distribution
  \[
  \lim_{N\to+\infty} \left(\frac{X^N_2(T^N_1)}{X^N_2(0)}, \frac{T^N_1}{X^N_2(0)^{p-1}}\right) \steq{dist} \left(U^{\delta_1}, E_1\right),
  \]
where $U$ is a uniform random variable on $[0,1]$, independent of $E_1$ an exponential random variable with parameter $r_1$ defined by Relation~\eqref{eqr1}, and 
\begin{equation}\label{eqmq1}
\delta_1\steq{def} \frac{\kappa_3(p{-}1)!}{\kappa_1}.
\end{equation}
\end{proposition}
\begin{proof}
Let $H_N$ be the hitting time of $p$ for $(X^N_1(t))$, $H_N$ has the same distribution as $S_Z(k, x_2^N)$.  Its asymptotic behavior is given by Lemma~\ref{TZlim}.
Since the reaction $pS_1{+}S_2{\rightleftharpoons} pS_1{+}2S_2$ cannot occur on the time interval $[0,H_N]$, we have $X^N_2(H_N){=}x^N_2{+}p{-}a\steq{def}x_2^{N,1}$.

Define $\tau_N$ as
\[
\tau_N=\inf\{t{>}0: X^N_1(H_N{+}t)=p{-}1\}.
\]
If the time origin is translated to $H_N$, by using the strong Markov property,  it is enough to  study the asymptotic behavior of $X^N_2(\tau_N)$ starting from $x_2^{N,1}$.

With the same arguments as in the proof of Proposition~\ref{ConvT2Prop}, external arrivals do not play a role on the time interval $[0,\tau_N)$ since the other reaction rates being of the order of $N$ or $N^2$. We can therefore assume that $X^N_1(s)$ is constant equal to $p$ until $\tau_N$. It is easily seen that the sequence of random variables $(N\tau_N)$ is tight. 

After time $0$, the transition $x{\to}x{-}e_2$ occurs until time $\nu_{1,N}$ when one of the reactions  $x{\to}x{-}e_1{-}e_2$ or $x{\to}x{+}e_2$ occurs. To study the random variable $X_2^N(\nu_{1,N})$, modulo a time change,  it is enough to consider the Markov process with $Q$-matrix
  \[
x\longrightarrow x{+}\begin{cases}
  -e_1{-}e_2& \kappa_1,\\
   e_2& \kappa_2 (p{-}1)!,\\
  {-}e_2& \kappa_3(p{-}1)!(x_2{-}1)^+.
xs\end{cases}
\]
If $F_1$ is an exponentially distributed random variable with parameter $\kappa_1{+}\kappa_2 (p{-}1)!$, the state of $(X_2(t))$ at $\nu_{1,N}$  is simply 
\[
X^N_2(\nu_{1,N}{-})\steq{dist} 1+\sum_{i{=}1}^{x_2^{N,1}-1}\ind{E_i{\ge} F_1},
\]
where $(E_i)$ is an i.i.d. sequence of exponential random variables with parameter $\kappa_3(p{-}1)!$, and $\left|X^N_2(\nu_{1,N})-X^N_2(\nu_{1,N}{-})\right|\leq 1$. For the convergence in distribution,
\[
\lim_{N\to+\infty}\frac{X^N_2(\nu_{1,N})}{X_2^N(0)}=\exp\left({-}\kappa_3(p{-}1)! F_1\right).
\]
The transition $x{\to}x{-}e_1{-}e_2$ occurs at time $\nu_{1,N}$ with probability $1{-}q_1$, with
\[
q_1{=}\frac{\kappa_2(p{-}1)!}{\kappa_1{+}\kappa_2(p{-}1)!},
\]
and in this case $\tau_N{=}\nu_{1,N}$. Otherwise, there is a new cycle of length $\nu_{2,N}$ and the relation
\[
\lim_{N\to+\infty}\frac{X^N_2(\nu_{1,N}){+}\nu_{2,N})}{X_2^N(0)}=\exp\left({-}\kappa_3(p{-}1)! (F_1{+}F_2)\right),
\]
holds, where $(F_i)$ is an  i.i.d. sequence with the same distribution as $F_1$. 
By induction we obtain the convergence in distribution
\[
\lim_{N\to+\infty}\frac{X^N_2(\tau_{N})}{X_2^N(0)}=\exp\left({-}\kappa_3(p{-}1)!\sum_{1}^{G}F_i\right),
\]
where $G$ is a random variable independent of $(F_i)$ with a geometric distribution with parameter $q_1$, $\P(G{\ge}n){=}q_1^{n-1}$ for $n{\ge}1$. Straightforward calculations give the desired representation. 
\end{proof}
In view of the last result it is natural to expect that the convergence of the scaled process $(X^N_2(t/N^{p-1})/N$ to  a Markov process with jumps. The only problem is that, as we have seen in the last proof, the jumps downward of the limit process are due to a large number of small jumps, of the order of $N$,  on the time interval of length $\tau_N$ of the previous proof. Event if $\tau_N$ is arbitrarily small when $N$ gets large, there cannot be convergence in the sense of the classical $J_1$-Skorohod topology. There are topologies on the space of \cadlag functions ${\cal D}(\R_+)$ for which convergence in distribution may hold in such a context. See~\citet{Jakubowski} for example. For the sake of simplicity, we present a convergence result formulated for a weaker topology expressed in terms of the occupation measure 

We now introduce a Markov process on $(0,1]$ as the plausible candidate for a limiting point of $(X^N_2(t/X^N_2(0)^{p-1})/N$.
\begin{definition}
The infinitesimal generator ${\cal A}$ of a Markov process $(U(t))$ on $(0,1]$ is defined by, for $f{\in}{\cal C}_c((0,1])$, 
\begin{equation}\label{GenA}
{\cal A}(f)(x)=\frac{r_1}{x^{p-1}}\int_0^1\left(f\left(xu^{\delta_1}\right){-}f(x)\right)\diff u, \quad x{\in}(0,1].
\end{equation}
where $r_1$, $\delta_1$ are constants defined by Relations~\eqref{eqr1} and~\eqref{eqmq1}.
\end{definition}
\begin{proposition}\label{MarkU}
A Markov process on $(0,1]$ with  infinitesimal generator  ${\cal A}$ defined by Relation~\eqref{GenA} is an explosive process converging almost surely to $0$. 
\end{proposition}
\begin{proof}
Let $(U(t))$  be such a process and assume that $U(0){=}\alpha{\in}(0,1]$.  By induction, the  sequence of states visited by the process has the same distribution as $(V_n)$  with, for $n{\ge}0$, 
  \[
  V_n\steq{def} \alpha \exp\left({-}\delta_1\sum_{i=1}^nE_i\right),
  \]
  where $(E_i)$ is an i.i.d. sequence of exponentially distributed random variables with parameter $1$. The sequence of the instants of  jumps has the same distribution as 
  \[
\left(t^V_n\right)\steq{def}\left(\sum_{i=1}^{n} (V_{i-1})^{p-1}\frac{G_i}{r_1}\right),
  \]
	where $(G_i)$ is an i.i.d. sequence of exponentially distributed random variables with parameter $1$, independent of $(E_i)$.
 It is easily seen that $(V_n)$ converges to $0$ almost surely and that the sequence $(\E(t^V_n))$ has a finite limit. The proposition is proved. 
\end{proof}

\begin{definition}[Scaled occupation measure of $(X^N_2(t))$]\label{defOcc}
For $N{\ge}1$,  $\Lambda_N$ is the random measure on $\R_+{\times}(0,1]$ defined by, for $f{\in}{\cal C}_c(\R_+{\times}(0,1])$, 
\begin{equation}\label{OccX2}
  \croc{\Lambda_N,f}=\frac{1}{N^{p-1}}\int_0^{+\infty}f\left(\frac{s}{N^{p-1}},\frac{X^N_2(s)}{N}\right)\diff s.
\end{equation}
\end{definition}
We can now state our main scaling result for large initial states near the vertical axis.
\begin{theorem}
If $(X_N(t))$ is the Markov process  associated to the CRN~\eqref{CapCRN} whose initial state is  $(a, x^N_2){\in}\N^2$, $a{\le}p{-}1$,  and  such that 
  \[
  \lim_{N\to+\infty}{x_2^N}/{N}=\alpha{>}0,
  \]
then the sequence $(\Lambda_N)$ defined by Relation~\eqref{OccX2} converges in distribution to $\Lambda_\infty$, the occupation measure of  $(U(t))$  a Markov process with    infinitesimal generator  ${\cal A}$ defined by Relation~\eqref{GenA}  starting at $\alpha$,  i.e. for $f{\in}{\cal C}_c(\R_+{\times}(0,1])$, 
\[
    \croc{\Lambda_\infty,f}=\int_0^{+\infty} f(s,U(s))\diff s.
  \]
\end{theorem}
\begin{proof}
Without loss of generality, due to the multiplicative properties of the convergence, see Proposition~\ref{Decay}, we can take $\alpha{=}1$ and assume that $X^N_2(0){=}N$.  Recall that the Laplace transform of a random measure  $\Lambda$ on $\R_+{\times}(0,1]$ is given by
\[
{\cal L}_\Lambda(f)\steq{def}\E\left(\exp\left(-\croc{\Lambda,f}\right)\right),
\]
for  a non-negative function  $f{\in}{\cal C}_c(\R_+{\times}(0,1])$. See Section~3 of~\citet{Dawson}. 

To prove the convergence in distribution of $(\Lambda_N)$ to $\Lambda_\infty$, it is enough to show that the convergence
  \[
  \lim_{N\to+\infty} {\cal L}_{\Lambda_N}(f)={\cal L}_{\Lambda_\infty}(f),
  \]
holds for all non-negative functions  $f{\in}{\cal C}_c(\R_+{\times}(0,1])$. See  Theorem~3.2.6 of~\cite{Dawson} for example. 

  If $f{\in}{\cal C}_c(\R_+{\times}(0,1])$, its support  is included in some $[0,T]{\times}(\eta,1]$, for $\eta{>}0$ and $T{>}0$. Let $(T^N_k)$ the sequence of stopping times defined by Relation~\eqref{eqTk}. The Laplace transform of $\Lambda_N$ at $f$ is given by
\begin{equation}\label{LapN}
{\cal L}_{\Lambda_N}(f)=\E\left(\exp\left({-}\sum_{k{\ge}0}\int_{T^N_k/N^{p-1}}^{T^N_{k+1}/N^{p-1}}f\left(s, \frac{X^N_2\left(T^N_k\right)}{X^N_2(0)}\right)\diff s\right)\right).
\end{equation}
    
As defined in the proof of Proposition~\ref{MarkU}, let $(t^V_k,V_k)$ be the sequence of couples of instants of jumps and the value of the Markov process $(V(t))$ at that time.
For $\eps{>}0$, there exists some $n_0$ such that
\[
\P\left(\alpha\prod_{i=1}^{n_0}V_i\ge \frac{\eta}{2}\right)\le \eps/2,
\]
holds, and, consequently, 
\begin{equation}\label{LapUep}
\left|{\cal L}_\Lambda(f)-\E\left(\exp\left({-}\sum_{k=0}^{n_0{-}1}\int_{t^V_k}^{t^V_{k+1}}f\left(s, V_k\right)\diff s\right)\right)\right|\le \eps.
\end{equation}
Proposition~\ref{Decay} shows that, for the convergence in distribution, 
\[
\lim_{N\to+\infty} \left(\frac{X^N_2(T^N_{k+1})}{X^N_2(T^N_{k})},\frac{T^N_{k+1}{-}T^N_k}{X^N_2(T^N_k)^{p-1}}, k{\ge}0\right)=\left(U_k^{\delta_1},E_k,k{\ge}0\right),
\]
where $(U_k)$ and $(E_k)$ are i.i.d. independent sequences of random variables whose respective distributions are uniform on $[0,1]$, and exponentially distributed on $\R_+$ with parameter $r_1$. Hence, there exists  $N_0$ such that if $N{\ge}N_0$, then
\begin{equation}\label{LapXep}
  \begin{cases}
\displaystyle    \left|{\cal L}_{\Lambda_N}(f)-\E\left(\exp\left({-}\sum_{k=0}^{n_0-1}\int_{T^N_k/N^{p-1}}^{T^N_{k+1}/N^{p-1}}f\left(s, \frac{X^N_2\left(T^N_k\right)}{X^N_2(0)}\right)\diff s\right)\right)\right|\le 2\eps,\\
\displaystyle    \P\left(\frac{X^N_2(T^N_{k+1})}{X^N_2(T^N_{k})}{\le}1, \forall k{\in}\{0,\ldots,n_0\}\right)\ge 1{-}\eps.
  \end{cases}
\end{equation}
Define, for $n{>}0$,
\[
 (I^N_n)\steq{def} \left(\sum_{k=0}^{n-1}\int_{T^N_k/N^{p-1}}^{T^N_{k+1}/N^{p-1}}f\left(s, \frac{X^N_2\left(T^N_k\right)}{X^N_2(0)}\right)\diff s \right)
\]
In views of Relations~\eqref{LapUep} and~\eqref{LapXep}, all we have to do is to prove that, for every $n{>}0$,  the convergence in law of $(I^N_n)$  to
\[
I_n\steq{def} \int_{0}^{t^V_{n}}f\left(s, V(s)\right)\diff s=\sum_{k=0}^{n-1}\int_{t^V_{k}}^{t^V_{k+1}}f\left(s, V_k\right)\diff s,
\]
as $N$ gets large.

We will prove by induction on $n{>}0$, the convergence in distribution 
\begin{multline*}
\lim_{N\to+\infty}\left(I^N_n, \left|\ln\left(\frac{X^N_2\left(T^N_n\right)}{X^N_2(0)}\right)\right|, \frac{T^N_n}{X^N_2(0)^{p-1}}\right)
\\=\left(\int_{0}^{t^V_{n}}f\left(s, V(s)\right)\diff s,\left|\ln\left(V_n\right)\right|,t^V_{n}\right).
\end{multline*}
We will show the convergence of the Laplace transform of the three random variables taken at $(a,b,c)$, for $a$, $b$, $c{>}0$. 

For $n=1$, this is  direct consequence of Proposition~\ref{Decay}.
If it holds for $n{\ge}1$,  the strong Markov property of $(X^N(t))$ for the stopping time $T^N_n$ gives the relation
\begin{multline*}
H_N(a,b,c)\steq{def}\E\left(\left.\exp\left({-}a I^N_{n+1}{-}b\left|\ln\left(\frac{X^N_2\left(T^N_{n+1}\right)}{X^N_2(0)}\right)\right|{-}c\frac{T^N_{n+1}}{X^N_2(0)^{p-1}}\right)\right|{\cal F}_{T^N_{n}}\right)
\\=
\exp\left({-}a I^N_{n}{-}b\left|\ln\left(\frac{X^N_2\left(T^N_{n}\right)}{X^N_2(0)}\right)\right|{-}c\frac{T^N_{n}}{X^N_2(0)^{p-1}}\right)\\\times \Psi_N\left(\frac{X^N_2\left(T^N_n\right)}{X^N_2(0)}, \frac{T^N_n}{X^N_2(0)^{p-1}}\right),
\end{multline*}
where, for $x{>}0$ and $u{>}0$, we define
\begin{multline*}
  \Psi_N\left(x,u\right)\steq{def} \E_{(p{-}1,\lfloor Nx\rfloor)}\left[\exp\left({-}a\int_0^{T^N_1{/}X^N_2(0)^{p-1}}\hspace{-12mm}f\left(s{+}u,x\right)\diff s\right.\right.\\
\left.\left.      {-}b\left|\ln\left(\frac{X^N_2\left(T^N_{1}\right)}{X^N_2(0)}\right)\right|{-}c\frac{T^N_{1}}{X^N_2(0)^{p-1}}\right)\right].
\end{multline*}
Proposition~\ref{Decay}, and the fact that the sequence $(N\tau_N)$  is tight in the proof of this proposition, gives the convergence
\begin{multline*}
\lim_{N\to+\infty} \Psi_N\left(x,u\right)\\=
\E_{x}\left(\exp\left({-}a\int_0^{E_{n+1}}f\left(s{+}u,x\right)\right)\diff s{-}b\left|\ln\left(U_{n+1}^\delta\right)\right|{-}cE_{n+1}\right),
\end{multline*}
where $U_{n+1}$ is a uniform random variable on $[0,1]$, independent of $E_{n+1}$ an exponential random variable with parameter $r_1$.   
With the induction hypothesis for $n$, Lebesgue's Theorem and the strong Markov property of $(U(t))$, we obtain the convergence
\begin{multline*}
  \lim_{N\to+\infty} \E(H_N(a,b,c))=
E\left[\rule{0mm}{6mm} \exp\left({-}a I_n{-}b\left|\ln V_n \right|{-}ct^V_{n}\right)\right.\\
  \left.{\times}\exp\left({-}a\int_{t^V_{n}}^{t^V_{n+1}}f\left(s,x\right)\diff s{-}b\left|\ln\left(\frac{V_{n+1}}{V_n}\right)\right|{-}c\left(t^V_{n+1}{-}t^V_n\right)\right)\right]
\\
=\E\left(\exp\left({-}a I_{n+1}{-}b\left|\ln V_{n+1} \right|{-}ct^V_{n+1}\right)\right).
\end{multline*}
The theorem is proved.
\end{proof}
\printbibliography

\appendix

\addcontentsline{toc}{section}{Appendix}
\addtocontents{toc}{\protect\setcounter{tocdepth}{0}}
\section{Technical Proofs}\label{AppProof}
The proofs of this section, although not difficult, are detailed for the sake of completeness, and also to show that some ingredients of a scaling analysis are essentially elementary. For basic results on martingale theory and classical stochastic calculus, see~\citet{Rogers2}.

To investigate scaling properties of stochastic CRNs, the formulation in terms of stochastic differential equations (SDE) to describe the Markov process is used. We first recall briefly the technical framework.

\subsection{A SDE  Formulation of CRN  Markov Processes}\label{SDESec}
The Markov process with $Q$-matrix defined by Relation~\eqref{00Qmat} can be classically expressed as the solution of a  martingale problem.  See Theorem~(20.6) in Section~IV of~\citet{Rogers2}.

We assume that on the probability space we have   a set of independent Poisson point processes ${\cal P}_r$, $r{\in}{\cal R}$ on $\R_+^2$   with intensity measure the Lebesgue measure on $\R_+^2$. See~\citet{Kingman}. The Markov process has the same distribution as the solution  $(X(t)){=}(X_i(t))$ of the SDE,
\begin{equation}\label{SDECRN}
\diff X(t)=\sum_{r{=}(y_r^-,y_r^+){\in}{\cal R}} \left(y_r^+{-}y_r^-\right){\cal P}_{r}\left(\left(0,\kappa_r \frac{X(t{-})!}{(X(t{-}){-}y_r^-)!}\right),\diff t\right),
\end{equation}
with the notation, for $a{\ge}0$,
\[
{\cal P}_r((0,a),\diff t) =\int_{s=0}^a{\cal P}_r(\diff s,\diff t).
\]
Note that a solution of  SDE~\eqref{SDECRN} is not, a priori,  defined on the entire half-line  in the case of an explosive process, i.e. when the instants of jumps of the process converge to some finite random variable $T_\infty$. In this case, the convention is that a point $\dag$ is added to the state space and  $X(t)$ is defined as $\dag$  for all $t{\ge}T_\infty$. 

The associated filtration is $({\cal F}_t)$, where, for $t{\ge}0$,  ${\cal F}_t$ is the completed  $\sigma$-field generated by the random variables
\begin{equation}\label{SDEFilt}
{\cal F}_t\steq{def} \sigma\left({\cal P}_r(A{\times}[0,s)), r{\in}{\cal R}, s{\le}t, A{\in}{\cal B}(\R_+)\right). 
\end{equation}

{\em A comment on the use of Poisson processes on $\R_+^2$. }  If $(\lambda(t))$ is a \cadlag adapted process, a Poisson point process with intensity $(\lambda(t))$ can be represented in two ways:
\begin{enumerate}
\item Following Kurtz, see~\citet{KurtzEthier}, if ${\cal N}$ is a Poisson point process with rate $1$ on $\R_+$, the counting  measure of such a point process can be expressed as
  \[
({\cal A}(t))\steq{def}  \left({\cal N}\left(0,\int_0^t\lambda(s)\diff s\right)\right).
  \]
\item In our paper we take the representation
  \[
({\cal B}(t))\steq{def}    \left(\int_0^t{\cal P}((0,\lambda(s-))),\diff s\right),
  \]
  where ${\cal P}$ is a Poisson process on $\R_+^2$ with intensity measure $\diff s{\otimes}\diff t$. 
\end{enumerate}
It is not difficult to see that $({\cal A}(t))$ and $({\cal B}(t))$ have the same distribution.

It should be noted that the filtration $({\cal F}_t)$ we have defined is dependent of the process $(\lambda(t))$.
 A filtration for $({\cal A}(t))$ would a priori depend on  $(\lambda(t))$.  When  coupling  constructions are considered, there may be  different such processes $(\lambda(t))$, with  a common  driving Poisson process. The definition of the filtration, which is crucial for martingale, stopping time properties, is not impossible in this case, but may be quite cumbersome to define properly.

 \bigskip

 Provided that $(X(t))$ is well defined on $[0,T]$, $T{>}0$, the integration of SDE~\eqref{SDECRN} gives the relation
\begin{equation}\label{SDECRNInt}
X(t)= X(0){+}\sum_{r{\in}{\cal R}}M_r(t){+}\sum_{r{\in}{\cal R}}\kappa_r\left(y_r^+{-}y_r^-\right) \int_0^t  \frac{X(s)!}{(X(s){-}y_r^-)!}\diff s
\end{equation}
on the time interval $[0,T]$, where, for $r{\in}{\cal R}$,  $(M_r(t)){=}(M_{r,i}(t))$ is a local martingale defined by
\begin{equation}\label{MartCRN}
\left(\left(y_r^+{-}y_r^-\right)\int_0^t \left({\cal P}_{r}\left(\left(0,\kappa_r \frac{X(s{-})!}{(X(s{-}){-}y_r^-)!}\right),\diff s\right){-}\kappa_r\frac{X(s)!}{(X(s){-}y_r^-)!}\diff s\right)\right),
\end{equation}
and its previsible increasing process is given  by, for $1{\le}i,j{\le}n$, 
\begin{equation}\label{CrocCRN}
\left(\croc{M_{r,i},M_{r,j}}(t)\right)=\left(\left(y_{r,i}^+{-}y_{r,i}^-\right)\left(y_{r,j}^+{-}y_{r,j}^-\right)\kappa_r\int_0^t \frac{X(s)!}{(X(s){-}y_r^-)!}\diff s\right).
\end{equation}

\end{document}